\documentclass[11pt]{amsart}

\usepackage[colorlinks = true,
            linkcolor = red,
            urlcolor  = magenta,
            citecolor = black,
            anchorcolor = blue]{hyperref}
\usepackage{color}

\usepackage[alphabetic, initials]{amsrefs}
\usepackage{upgreek}
\usepackage{tikz}
\usepackage[applemac]{inputenc} 
\usetikzlibrary{arrows}
\usepackage{xcolor}
\usepackage[all]{xy}
\usepackage{graphicx}
\usepackage{wrapfig}
\usepackage{yfonts}
\usepackage{graphicx}
\usepackage{amssymb}
\usepackage{epstopdf}
\usepackage{mathrsfs}
\usepackage[mathcal]{eucal}
\usepackage{bm}

\renewcommand{\eprint}[1]{\href{https://arxiv.org/abs/#1}{#1}}

\DeclareMathOperator{\GL}{\mathrm{GL}}

\DeclareMathOperator{\Hom}{Hom}

\newcommand{\hcL}{\hat{\cL}}

\BibSpec{article}{%
    +{}  {\PrintAuthors}                {author}
    +{,} { \textit}                     {title}
     +{,} {}                            {note}
    +{.} { }                            {part}
    +{:} { \textit}                     {subtitle}
    +{,} { \PrintContributions}         {contribution}
    +{.} { \PrintPartials}              {partial}
    +{,} { }                            {journal}
    +{}  { \textbf}                     {volume}
    +{}  { \PrintDate}                {date}
    +{,} { \issuetext}                  {number}
    +{,} { \eprintpages}                {pages}
    +{,} { }                            {status}
    +{,} { \DOI}                   {doi}
    +{,} { \eprint}        {eprint}
      +{,} {\publisher}                {publisher}
    +{,} { \address}                   {address}
    +{}  { \parenthesize}               {language}
    +{}  { \PrintTranslation}           {translation}
    +{;} { \PrintReprint}               {reprint}
    +{.} {}                             {transition}
    +{}  {\SentenceSpace \PrintReviews} {review}
}

\newtheorem{Thm}{Theorem}[section]
\newtheorem{Lem}[Thm]{Lemma}
\newtheorem{Prop}[Thm]{Proposition}
\newtheorem{Cor}[Thm]{Corollary}

\theoremstyle{definition}
\newtheorem{Def}[Thm]{Definition}
\theoremstyle{remark}
\newtheorem{Rem}[Thm]{Remark}

\newtheoremstyle{named}{}{}{\itshape}{}{\bfseries}{.}{.5em}{#1 #3}
\theoremstyle{named}

\def\C{\mathbb{C}}

\def\g{\mathfrak{g}}
\def\Frenkel:2013uda{\mathfrak{h}}

\def\cD{\mathcal{D}}

\def\cL{\mathcal{L}}

\def\cO{\mathcal{O}}

\def\cV{\mathcal{V}}
\def\cW{\mathcal{W}}

\def\bo{\textbf{o}}

\def\=>{\Longrightarrow}

\def\to{\longrightarrow}

\def\o+{\oplus}
\def\bo+{\bigoplus}

\def\<{\langle}
\def\>{\rangle}
\def\({\left(}
\def\){\right)}

\def\^{\wedge}
\def\+{\dagger}

\def\dd[#1,#2]{\frac{d#1}{d#2}}
\def\del[#1,#2]{\frac{\partial #1}{\partial #2}}
\def\over[#1]{\overline{#1}}
\def\vec[#1]{\overrightarrow{#1}}

\makeatletter
\def\mr@ignsp#1 {\ifx\:#1\@empty\else #1\expandafter\mr@ignsp\fi}%
\newcommand{\multiref}[1]{\begingroup
\xdef\mr@no@sparg{\expandafter\mr@ignsp#1 \: }%
\def\mr@comma{}%
\@for\mr@refs:=\mr@no@sparg\do{\mr@comma\def\mr@comma{,}\ref{\mr@refs}}%
\endgroup}
\makeatother

\newcommand{\hypref}[2]{\ifx\href\asklFrenkel:2013udaas #2\else\href{#1}{#2}\fi}

\tikzset{->-/.style={decoration={
  markings,
  mark=at position .5 with {\arrow{latex}}},postaction={decorate}}}
\tikzset{
    >=latex
    }

\newcommand{\nc}{\newcommand}
\nc{\on}{\operatorname}
\nc{\la}{\lambda}
\nc{\wh}{\widehat}
\nc{\ghat}{\wh\g}
\nc{\mb}{\mathbf}

\begin{document}
\title{The Zoo of Opers and Dualities}

\author[P. Koroteev]{Peter Koroteev}
\address{
Department of Mathematics,
University of California,
Berkeley, CA 94720, USA,
pkoroteev@berkeley.edu;
and 
NHETC, Rutgers University, 
Piscataway, NJ, 08854, USA
}

\author[A.M. Zeitlin]{Anton M. Zeitlin}
\address{
          Department of Mathematics, 
          Louisiana State University, 
          Baton Rouge, LA 70803, USA}

\date{\today}

\numberwithin{equation}{section}

\begin{abstract}
We investigate various spaces of $SL(r+1)$-opers and their deformations. 
For each type of such opers, we study the quantum/classical duality, which relates quantum integrable spin chains with classical solvable many body systems. In this 
context, quantum/classical dualities serve as an interplay between two different coordinate systems on the space of opers. 
We also establish correspondences between the underlying oper spaces, which recently had multiple incarnations in symplectic duality and bispectral duality.  
\end{abstract}

\maketitle 

\setcounter{tocdepth}{1}
\tableofcontents

\section{Introduction}\label{Sec:Intro}
\subsection{Integrable systems and enumerative geometry}
The study of integrable systems has led to a plethora of important discoveries in modern mathematics. 
In the 1960s and 1970s, the explosion of interest in classical integrable models in the context of soliton theory significantly impacted differential and algebraic geometry, as well as representation theory. Eventually, that became an area of its own, and the theory of integrable systems has become an independent field of research in modern mathematical physics. An enormous number of such systems were discovered in the 1960s and 1970s, some of which were infinite dimensional, like the original Korteweg-de Vries (KdV) system describing soliton waves in a channel, whereas others were finite-dimensional, originating in solid-state physics, corresponding to interacting particles on a lattice, like the celebrated Toda system.

In the 1980s, the study of quantum integrable models led to the discovery of quantum groups. Then, in the last decade of the 20th century, another area of modern algebraic geometry, known as enumerative geometry, was enriched by a multitude of new notions and tools with deep connections to modern theoretical physics, such as 
Gromov-Witten invariants, quantum cohomology, and mirror symmetry. The pioneering works of Dubrovin, Kontsevich, and Witten led to the discovery of hidden connections between the theory of classical integrable systems, including original equations of soliton theory, and enumerative invariants.

In order to understand, on a basic level, what classical and quantum integrable models are, consider a symplectic manifold of dimension $2n$ such that there are $n$ mutually commuting functions (Hamiltonians) 
under the corresponding Poisson structure. The Liouville-Arnold theorem states that the level surface defined by the corresponding Hamiltonians is a Lagrangian subvariety such that its connected components are 
diffeomorphic to the product of abelian tori and Euclidean space. The flows of the vector fields corresponding to the Hamiltonians create the \textit{angle} coordinates, which can be paired with \textit{action} coordinates, thereby forming a Darboux coordinate system on the corresponding connected component. `Solving' an integrable model entails finding the corresponding action-angle variables and the corresponding transformation map. This can be a complicated problem, particularly when an integrable system is infinite-dimensional, like the KdV system. In that particular example, the action-angle variables are related to the spectral data of the Sturm-Liouville operator, and their construction is the goal of the Inverse Scattering Method (ISM).

In quantum integrable models, the Poisson bracket is replaced by the commutator, and the Hamiltonians are replaced by mutually commuting set of operators in a Hilbert space. Instead of finding the action-angle variables, the task of interest is the simultaneous diagonalization of these operators. In the case of quantum spin chains, the Hilbert space is the tensor product of finite-dimensional modules of the Yangian $Y_{q}(\mathfrak{g})$ or affine quantum group $U_{q}(\widehat{\mathfrak{g}})$ or elliptic quantum group associated to $\mathfrak{g}$, where 
$\mathfrak{g}$ is a simple Lie algebra. The corresponding integrable models are known as XXX, XXZ, and XYZ spin chains respectively, and the XXX is the limit of the XXZ, which is in turn the limit of the XYZ. There is also an important `semiclassical limit', known as the Gaudin model, where the Hilbert space is the tensor product of finite-dimensional representations of $\mathfrak{g}$.

The transfer matrices, a commuting family of operators labeled by finite-dimensional representations, can be described through the associated braiding operators, called R-matrices. These R-matrices are an intrinsic part of the Yangian/quantum group structure and are associated with a regular semisimple element $Z=\prod_iz_i^{\check{\alpha}_i}$.
However, the case of the Gaudin model is different: 
the quantum Hamiltonians can be described explicitly in terms of generators of $\mathfrak{g}$, and the twist parameter is an element of the Cartan subalgebra of $\mathfrak{g}$. 
The original method of diagonalization of the transfer matrices was given by the algebraic Bethe Ansatz, introduced in the 1970s-1980s, which became a part of the quantum inverse scattering method \cite{Bogolyubov:1993qism,Reshetikhin:2010si} and historically led to the discovery of quantum groups. The result is a system of algebraic (Bethe) equations such that the eigenvalues of transfer matrices generate functions of the roots of these Bethe equations. However, although this method was quite successful for particular examples, it lacked universality and was too representation specific. At the time, these qualities largely led to a lack of conceptual-level mathematical understanding of what Bethe equations are.

The 1990s brought another approach to the algebraic Bethe Ansatz method, considering classical and quantum Knizhnik-Zamolodchikov (qKZ) equations. These are the differential/difference equations that govern the matrix elements of intertwining operators (correlation functions) for the affine algebras and their deformations discussed above. For example, the qKZ equations for the XXZ chain appear as $\Psi(\{\mathfrak{q}^{\sigma}a_i\})=S_{\sigma}\Psi(\{a_i\})$, where $\Psi$ takes values in the Hilbert space, and $\{a_i\}$ is the collective label for evaluation parameters for the modules contributing to the Hilbert space. On the left-hand side of the equation, these parameters are shifted by the integer degrees of $q$ controlled
by $\sigma$. The parameter $\mathfrak{q}$ is related to the central charge $k$ and 
deformation parameter $q$ as $\mathfrak{q}=q^{-(k+h^{\vee})}$, where $h^{\vee}$ is the dual Coxeter element of $\mathfrak{g}$. The operators  $S_{\sigma}$ are constructed using R-matrices and the twist element $Z$. 
Their limit at the critical level $S_{\sigma}|_{k\rightarrow -h^{\vee}}$ gives the appropriate 
transfer matrices. Thus, the critical level asymptotics of the solutions of the qKZ equations 
lead to the desired eigenvalue problem of the quantum Hamiltonians. 
The classical Knizhnik-Zamolodchikov equation and the eigenvalue problem for the Gaudin model can be obtained in the $q\rightarrow 0$ limit in the Yangian case. 
There is also a commuting system of equations on the twist parameter $Z$, i.e., $\Psi(\{\mathfrak{q}^{\rho}z_i\})=M_{\rho}\Psi(\{z_i\})$, known as the dynamical qKZ equations, which will be necessary for the following steps. The seminal work of Cherednik and Matsuo established the relationship between solutions of the qKZ equations and the quantization of multiparticle systems, showing that the eigenfunctions of quantum Hamiltonians of an appropriate multiparticle system can be constructed from the solutions of the qKZ equations. This can be viewed as the quantum version of the quantum/classical duality, which we discuss in detail below.

Papers by Givental and his collaborators in the 1990s and early 2000s explicitly established the relations between multiparticle systems and fundamental enumerative invariants within quantum cohomology/quantum K-theoretic computations. 
Two key results should be mentioned. First, Givental and Kim \cite{givental1995} showed that the quantum equivariant cohomology ring of the complete flag variety can be interpreted as the algebra of functions on the intersection of two Lagrangian cycles of the Toda lattice. 
At the same time, in the context of quantum K-theory of complete flag varieties \cite{2001math8105G}, 
J-functions, which count Drinfeld's version of quasimaps as weighted equivariant 
Euler characteristics on a quasimap moduli space, satisfy certain difference equations, which identify these J-functions with eigenfunctions of quantum Hamiltonians of the difference Toda (q-Toda) system. This latter result emphasized the importance of the difference equations from quantum integrable models in enumerative problems. 
Second, a groundbreaking result emerged from the study of 2D/3d gauge theories by 
Nekrasov and Shatashvili \cite{Nekrasov:2009ui,Nekrasov:2009uh}. 
Their results showed that Bethe equations for XXX and XXZ models can be interpreted as the relations in the quantum cohomology rings of certain varieties. 
These varieties describe the Higgs vacua of these theories and are well known to mathematicians as Nakajima quiver varieties \cite{Nakajima:1999hilb}, an example of symplectic resolutions. 
As it usually happens in physics, such deformed rings emerge from counting 
(via indices of some elliptic operators) of particular solutions of differential equations. In those cases, these are vortex solutions, which can be mathematically interpreted as quasimaps to quiver varieties -- these involve a collection of vector bundles on the projective line, induced by the quiver data, as well as their sections with some extra conditions.

At the same time, quiver varieties became well known in the context of geometric representation theory through the work of Nakajima, Vasserot, and many others \cite{Nakajima:2001qg,Vasserot:wo}, who constructed the geometric action of Yangians and quantum groups on localized equivariant cohomology/K-theory of quiver varieties. These spaces can be identified with tensor products of finite-dimensional modules of the corresponding algebraic object and thus, in our terms, with the Hilbert spaces of XXX and XXZ modules.

Combining the insights from the work of Nekrasov and Shatashvili, along with Givental's results on the importance of difference equations, and geometric representation theory's interpretations of Yangians and quantum groups, Okounkov and his collaborators \cites{Braverman:2010ei,2012arXiv1211.1287M,Okounkov:2015aa,Okounkov:2016sya} found that the analog of Givental's J-functions for quasimaps to quiver varieties satisfy the qKZ equations. The equivariant parameters are identified with evaluation parameters of finite-dimensional modules, and eigenvalues $\{z_i\}$ of the twist element are identified with K\"ahler parameters, which produce the weight parameters in curve counting. 

The proper formulation of the Nekrasov-Shatashvili conjecture 
implies that the eigenvalues of the quantum tautological classes are governed by the same asymptotics governed by qKZ equation. That leads to the fact that the eigenvalues of the multiplication by tautological classes are the symmetric functions of Bethe roots. This was proven for $A_{r+1}$ quiver varieties in \cite{Koroteev:2017aa}.

In general, the Bethe Ansatz equations describing the relations in the quantum K-theory ring for quiver varieties can be summarized using the $QQ$-system:
\begin{eqnarray}\label{qq-system}
&&{{\widetilde\xi_i}}Q^i_{-}({u})Q^i_{+}(q  {u})-{\xi_i}Q^i_{-}(q  {u})Q^i_{+}({u}) \nonumber \\
&&=\Lambda_i({u})\prod_{j\neq i}\Bigg[\prod^{-a_{ij}}_{k=1} Q^j_{+}(q^{b_{ij}^k}{u})\Bigg],\quad 
i=1,\dots,r.
\end{eqnarray}

This is a system of equations for $Q^i_{+}$, where the index $i$ runs through the quiver vertices, 
${\xi_i}$, ${{\widetilde\xi_i}}$ are certain monomials in $\{z_i^{\pm 1}\}$, while $b_{ij}=1$ if $i>j$ 
and  $b_{ij}=0$, if $j>i$. Here, polynomials $\Lambda_i({u})$ are monic polynomials whose roots coincide with the equivariant 
parameters corresponding to framing at that vertex, and polynomials $Q^i_{-}(u)$ are auxiliary. Assuming certain mild non-degeneracy conditions are satisfied, the roots of $Q^i_+(u)$ coincide with Bethe roots, and the polynomial $Q^i_+(u)$ coincides with the eigenvalues of the operator $\widehat{Q}^i_+(u)$ of quantum multiplication by $\sum^{r+1}_{k=1}u^k\widehat{\Lambda^k\mathcal{V}_i}$, where $\mathcal{V}_i$ are quantum tautological bundles. The operators $\widehat{Q}^i_+(u)$ have interesting interpretations in terms of transfer matrices. In the case of quantum groups, they are known as Baxter operators, corresponding to the infinite-dimensional representations of the Borel subalgebra of $U_{q}(\widehat{\mathfrak{g}})$, which uses the fact that the R-matrix belongs to the completed tensor product of the opposite Borel subalgebras. These representations are known as {\it prefundamental} and in the $U_{q}(\widehat{\mathfrak{sl}(2)})$ case were introduced by Bazhanov, Lukyanov and Zamolodchikov in \cite{Bazhanov:1998dq} and studied in general by Frenkel and Hernandez \cite{Frenkel:2013uda,Frenkel:2016} following earlier results of Jimbo and Hernandez \cite{HJ}. Similar construction and the analog of the $QQ$-system should also exist for Yangians, with some progress being made in \cite{BFLMS}. 

Below we discuss another geometric interpretation of the $QQ$-system.

\subsection{Twisted $(^LG,q)$, $(^LG,\epsilon)$-opers and dualities}
A well-known example of the geometric Langlands correspondence, studied in \cite{Feigin:1994in,Frenkel:2003qx,Frenkel:2004qy,FFTL:2010,Rybnikov:2010,MR2452191,MR2525775} can be formulated as a one-to-one correspondence between the spectrum of Gaudin models associated with Lie algebra $\mathfrak{g}$ and Miura $^LG$-oper connections on a projective line with trivial monodromy, regular singularities, and double pole irregular singularity at infinity. By $^LG$ we refer to the simply connected group associated with $^L\mathfrak{g}$. 
This allows the following reformulation of the duality \cite{Brinson:2021ww} -- there is a one-to-one correspondence between the solutions of the differential form of the $QQ$-system (\ref{qq-system}) (obtained in the limit $q\rightarrow 1$) and the relevant oper connections. At the same time, assuming certain mild non-degeneracy
conditions, the solutions to this differential $QQ$-system are in one-to-one correspondence with Bethe equations of Gaudin models. 

This correspondence has been recently successfully deformed to the case of XXX and XXZ models \cite{KSZ,Frenkel:2020}. Namely, 
there is a deformed analog of a connection, which is a meromorphic section $A\in \text{Hom}_{\mathcal{O}(\mathbb{P}^1)}(\mathcal{F}_{^LG}, \mathcal{F}_{^LG}^{q})$, where $\mathcal{F}_{^LG}$ is a principal $^LG$-bundle, $\mathcal{F}^{q}_{^LG}$ is a pull-back bundle with respect to either additive or multiplicative action with respect to parameter $q$. The Miura oper condition is related to two reductions of $\mathcal{F}_{^LG}$ to $\mathcal{F}_{^LB_{\pm}}$ to Borel subgroups $^LB_{\pm}$. The $(^LG,q)$-oper condition means that $A$ belongs to the Coxeter cell with respect to $B_-$, while the Miura condition implies that $A$ preserves $\mathcal{F}_{^LB_{+}}$. Locally, that means 
$$
A(v)=\prod^r_{i=1}g_i^{\check{\alpha}_i}(v)e^{\frac{\Lambda_i(v)}{g_i(v)}e_i}\,,
$$
where $g_i(z), \Lambda_i(z)$ are rational functions. The condition that $A$ has regular singularities means that $\{\Lambda_i(v)\}_{i=1,\dots, r}$ are polynomials. The zero monodromy condition and double pole singularity at infinity are modified to the condition that the $(^LG,q)$-oper connection is gauge equivalent to $Z\in ^L\!\!H$, i.e. $A(v)=U(q v)ZU(v)^{-1}$. 

For a q-oper connection $A(v)$ we can associate a difference equation $f(qv)=A(v)f(v)$ which is a q-difference analogue of ordinary differential equations which arise in the study of opers.

Assuming certain mild non-degeneracy conditions on the corresponding objects, we proved \cite{Frenkel:2020,KoroteevZeitlinCrelle} that there is a one-to-one correspondence between nondegenerate $Z$-twisted Miura $(^L G,q)$-opers with regular singularities and the nondegenerate polynomial solutions of the $QQ$-systems (\ref{qq-system}) for a specific choice of parameters $\{b_{ij}^k\}$, associated with Lie algebra $\mathfrak{g}$.

In the case of simply-laced $\mathfrak{g}$, the $QQ$-systems from \cite{Frenkel:2020} are equivalent to standard Bethe Ansatz equations. However, the non-simply laced case is more involved (see the discussion in \cite{Frenkel:2020} and in \cite{Frenkel:2021vv}). In particular, given the relationship between the $QQ$-systems and the Bethe Ansatz equations, one obtains the correspondence between the space of functions on $Z$-twisted  $(^LG,q)$-opers with regular singularities as quantum cohomology/quantum K-theory ring on the corresponding quivers of ADE type.

A key idea of \cite{KoroteevZeitlinCrelle}, which allows proving the theorem above in full, following the partial result in \cite{Frenkel:2020} (which uses the intermediate object called the Miura-Pl\"ucker oper and heavy non-degeneracy conditions) is to introduce the notion of $(^LG,q)$-Wronskian: the generalization of the $q$-difference version of the Wronskian matrix, which was previously used to define $(SL(r+1),q)$-opers via associated bundles \cite{KSZ}.
$(^LG,q)$-Wronskians are the meromorphic sections of an $^LG$-bundle on the projective line, which satisfy a certain $q$-difference equation. We establish a one-to-one correspondence between $(^LG,q)$-opers and $(^LG,q)$-Wronskians in \cite{KoroteevZeitlinCrelle}. The elements of the $QQ$-system are identified with the certain  {\it generalized minors} of Berenstein, Fomin, and Zelevinsky \cite{Berenstein_1997,BERENSTEIN199649,FZ} of the $(^LG,q)$-generalized Wronskian, while the equations of the $QQ$-system emerge as the relations between the generalized minors discovered in \cite{FZ} which appear in the study of double Bruhat cells in the combinatorial context of cluster algebras.

In the case of $^LG=SL(r+1)$-opers, the above Wronskians can be explicitly written as determinants of matrices formed by components of a section of a line bundle and its q-shifts. The roots of polynomials inside such $(SL(r+1), q)$-Wronskian matrix provide an alternative set of coordinates for $(SL(r+1),q)$-opers. 
They have a special meaning. These coordinates provide one side of the correspondence between XXX/XXZ models and rational/trigonometric Ruijsennars-Schenider (rRS/tRS) many-body systems and offer an example of quantum/classical duality \cites{Mukhin:2009aa,Gaiotto:2013bwa,Mukhin:2012aa,Mironov:2012ba,Zabrodin:2017td,Zabrodin:2017vt,Zabrodin:}. Namely, suppose one identifies roots of polynomials from $(SL(r+1),q)$-Wronskian with the momenta and $z_i$ parameters with the values of coordinates for the tRS/rRS system. In that case, one obtains \cite{Koroteev:2017aa,Koroteev:2023aa} that 
quantum K-theory ring for quiver varieties of type $A$ is isomorphic to the ring of functions on the intersection of two Lagrangian cycles in the phase of the tRS integrable system.

The importance of this construction in a larger geometric context is as follows. The study of the so-called 3d mirror symmetry or symplectic duality is one of the central topics in modern mathematics, despite still having a vague mathematical definition. This pursuit aims to relate pairs of symplectic varieties, which appear as the so-called Higgs and Coulomb branches of supersymmetric 3d quiver gauge theories with eight supercharges so that specific collections of structures on these varieties coincide. For example, one such structure is the spectra of quantum K-theory rings, which describe the chiral rings (holomorphic operators) for the corresponding gauge theories. On the level of the $q$-difference equations discussed above, on the dual variety, the equivariant and K\"ahler 
parameters do interchange so that qKZ and the corresponding dynamical equations switch roles. There is no systematic theory of the construction of such pairs, but some examples are known. The above-mentioned quantum/classical duality leads to simple explicit formulas for the variable transformations for the quantum K-theory ring of the quantum cotangent bundle of complete flags, 
explicitly proving its self-duality on the level of quantum K-theory rings. On the integrable systems side, this duality was long studied on the level of spectra of spin chain models and is known as {\it bispectral duality} \cite{MR2409414,MR2641196,Bulycheva:2012aa,Gaiotto:2013bwa}. In \cite{Koroteev:2023aa}, we studied 3d mirror symmetry in the case of $A$-type quivers and extended this construction to cyclic quivers using techniques from \cite{koroteev_zeitlin_2023} and ideas of string/gauge theory of \cite{Gaiotto:2008ak,Gaiotto:2008sa,Gaiotto:2013bwa}. 

In this paper we discuss various limits of this construction, first by reducing our analysis to quantum cohomology: there the dynamical equation for the qKZ system becomes differential. That leads to the correspondence between the trigonometric Gaudin integrable system (from the asymptotics of the dynamical part) and XXX model (from the asymptotics of the difference equation). On the level of many body systems that lead to the duality between rational Ruijsenaars-Schneider (rRS) models and trigonometric Calogero-Moser (tCM) systems, which are dual on the level of quantum/classical duality to the tGaudin and XXX models. 

The oper formulation of the related limits of the $QQ$-systems works as follows. 
The $QQ$-system for the XXX model differs by replacing the multiplicative $q$-action with the additive $\epsilon$-action leading to what we call $Z$-twisted $(G, \epsilon)$-opers, which is very much similar to $Z$-twisted $(G, q)$-opers. 

Whereas for the tGaudin model, the situation is different -- we deal with opers as connections, however, the $Z$-twisted condition is imposed in a peculiar way. Instead of $\mathbb{P}^1$ we work on a cylinder $\mathbb{P}^1\backslash\{0,\infty\}$ so the connection will exhibit singular behavior while approaching the boundaries. In the standard coordinates on $\mathbb{P}^1$ the residue of the oper connection at $0$ and $\infty$ is equal to $Z\in\mathfrak{h}$.

One can consider a classical limit of this construction as well, which does not have an enumerative interpretation. That leads to the quantum/classical duality between rational Gaudin models and 
rational Calogero-Moser spaces, which are self-dual with respect to the limit of symplectic/bispectral duality. 

\subsection{Zoo of Opers}
Figure \ref{fig:EtingofDiamond} combines the results on the family of opers as well as the corresponding quantum and classical integrable systems which appear in the study of each type of opers. In the main body of the paper, we discuss each corner of this `diamond'.
\begin{figure}
\includegraphics[scale=0.4]{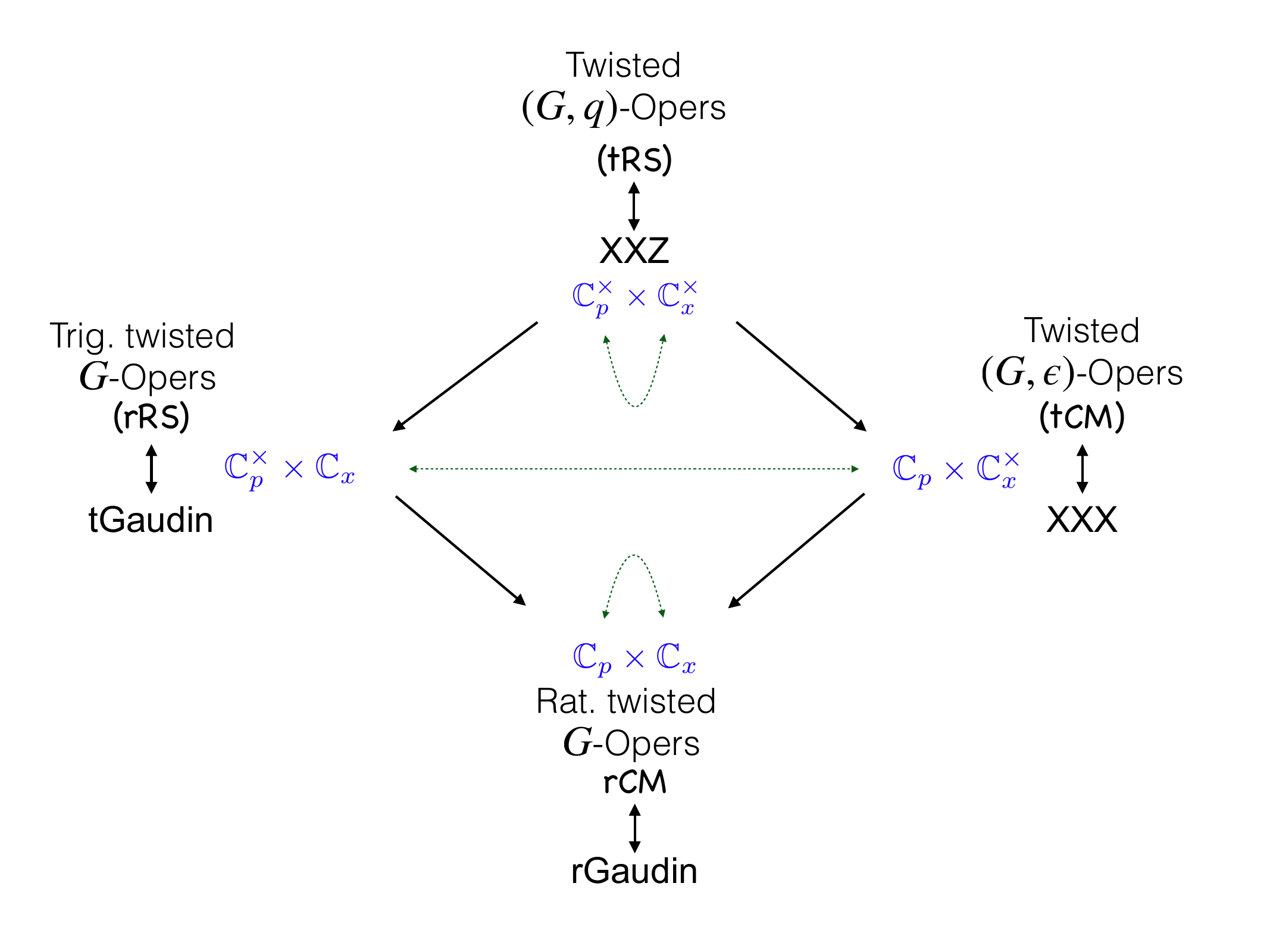}
\caption{The network of dualities between various types of opers and related integrable systems. Short vertical lines are the quantum/classical dualities, diagonal arrows show the double scaling limits between the models, while dashed lines designate the action of symplectic/bispectral dualities. The momenta $p$ and coordinates $x$ of the many body systems may take values in $\mathbb{C}^\times$ or $\mathbb{C}$ which is displayed in the figure.}
\label{fig:EtingofDiamond}
\end{figure}

Here the acronyms stand for the following integrable systems. Classical models are \textit{trigonometric Ruijsenaars-Schneider (tRS), rational Ruijsenaars-Schneider (rRS), trigonometric Calogero-Moser (tCM), rational Calogero-Moser (rCM)}. Meanwhile, the quantum spin chains are referred to as \textit{XXZ, XXX}-- Heisenberg spin chains with and without aniso\-tropies, \textit{trigonometric and rational Gaudin models -- tGaudin and rGaudin} respectively. All the above quantum systems can be solved using Bethe Ansatz.

\subsection{Elliptic Integrable Systems}
The $2\times 2$ diamond in Figure \ref{fig:EtingofDiamond} can be extended to a $3\times 3$ diamond which would include elliptic integrable models like the double elliptic integrable system (DELL) \cite{Koroteev:2019tz} as well as its degenerations -- the elliptic Ruijsenaars-Schneider system, the elliptic Calogero-Moser system, as well as their bispectral duals. We expect to find the corresponding elliptic generalizations of the space of opers on $\mathbb{P}^1$ and on the elliptic curve $\mathcal{E}$ which will describe the space of solutions of the novel elliptic $QQ$-systems. As of this writing, the bispectral dual of the DELL system is not known.

\subsection{Structure of the Paper}
In Section \ref{Sec:qOpers}, we study the top and right corners of the diamond which correspond to $Z$-twisted $(SL(r+1),q)$-opers and $(SL(r+1),\epsilon)$-opers respectively. Next, in Section \ref{Sec:DiffOpers}, we address the bottom and right corners where differential $SL(r+1)$-opers on $\mathbb{P}^1$ are discussed, one is gauge equivalent to constant regular semisimple element, another to a simple polar connection with residue given by regular semisimple element. We call these $SL(r+1)$-opers respectively as rationally $Z$-twisted and trigonometrically $Z$-twisted.   

In both sections, we prove the respective quantum/classical dualities between the space of opers and the $QQ$-systems (or the $qq$-systems for the differential opers). We demonstrate, for each corner of the diamond, that the conditions for the existence of the corresponding canonical nondegenerate opers provide the recipe to compute the Lax matrices for the related integrable systems. In the final Section \ref{Sec:CMSpace}, we provide the algebraic description of the Calogero-Moser space which can be used in deriving the trigonometric Ruijsenaars-Schneider Hamiltonians. Then we consider three different double-scaling limits which will land us on the remaining three corners of the diamond in Figure \ref{fig:EtingofDiamond}. We formulate the theorems on the bispectral duality between the systems involved.

\subsection{Acknowledgements}
P.K. is partlially supported by the DOE under grant DOE-SC0010008 to Rutgers and in part by the AMS-Simons travel grant. A.M.Z. is partially supported by Simons Collaboration Grant 578501 and NSF grant DMS-2203823

\section{$Z$-twisted $(SL(r+1),q)$-opers and $(SL(r+1),\epsilon)$-opers with regular singularities}\label{Sec:qOpers}
In this Section we will describe the difference oper structures we will be considering in this article. 
That will involve $Z$-twisted $(SL(r+1),q)$-opers and $(SL(r+1),\epsilon)$-opers on $\mathbb{P}^1$, both a difference versions of a connection with certain properties, related to two types of torus action on $\mathbb{P}^1$: multiplicative and additive correspondingly. 
In the next section we will discuss $Z$-twisted  Miura $SL(r+1)$ oper connections, one is again on $\mathbb{P}^1$, while another considered on a punctured disk, or, as we conveniently represent it, on a cylinder.

Each of the objects in this hierarchy can be thought of as a certain double-scaling limit of each other, which we will discuss as well. 

The exposition in this section will follow \cite{Koroteev:2018a} and \cite{koroteev_zeitlin_2023}.

\subsection{Miura $(SL(r+1),q)$-opers and $(SL(r+1),\epsilon)$-opers}
Let us consider the automorphisms $M_q: \mathbb{P}^1 \to \mathbb{P}^1$ and $M_{\epsilon}: \mathbb{P}^1 \to \mathbb{P}^1$ sending $z\to qz$, $z \mapsto z+\epsilon$, correspondingly, where $q\in\C^\times$ and $\mathbb{\epsilon}\in \C$.

In this Section, we will formulate all basic definitions for $(SL(r+1),q)$-opers. The corresponding definitions for $(SL(r+1),\epsilon)$-opers one obtains by replacing $M_q$ by $M_{\epsilon}$.

\begin{Def}    \label{qopflagD}
Let $U \subset \mathbb{P}^1$ be a Zariski open dense subset and let $V=U \cap M_q^{-1}(U)$.
 A meromorphic $(GL(r+1),q)$-{\em oper} on $\mathbb{P}^1$ is a triple $(\mathcal{A},E, \mathcal{L}_{\bullet})$, where $E$ is a vector bundle of rank $r+1$ on $\mathbb{P}^1$ and $\mathcal{L}_{\bullet}$ is the corresponding complete flag of the vector bundles, 
  $$\mathcal{L}_{r+1}\subset ...\subset \mathcal{L}_{i+1}\subset\mathcal{L}_i\subset\mathcal{L}_{i-1}\subset...\subset \mathcal{L}_1=E,$$ 
  where $\mathcal{L}_{r+1}$ is a line bundle, so that the meromorphic $(SL(r+1),q)$-connection 
  $\mathcal{A}\in \Hom_{\cO_{U}}(E,E^q)$, where $E^q$ is the pullback of $E$ under $M_q$,
  satisfies the following conditions:\\ 
i) $\mathcal{A}\cdot \mathcal{L}_i\subset \mathcal{L}^q_{i-1} $,\\
ii)  The restriction of $\mathcal{A}\in \text{Hom}(\mathcal{L}_{\bullet}, \mathcal{L}^q_{\bullet})$ to $V$ is invertible and satisfies the condition that the induced maps
  $$\bar{\mathcal{A}}_i:\mathcal{L}_{i}/\mathcal{L}_{i+1}\to \mathcal{L}^q_{i-1}/\mathcal{L}^q_{i},\qquad i = 2,\dots,r$$ are isomorphisms on $V$. \\
  An $(SL(r+1),q)$-$oper$ is a $(GL(r+1),q)$-oper with the condition that $det(\mathcal{A})=1$ on $U \cap M_q^{-1}(U)$.
\end{Def}

Note, that changing the trivialization of $E$ via $g(z) \in SL(r+1)(z)$ changes
$A(z)$ by the following $q$-gauge transformation
\begin{equation}   
\label{gauge tr}
A(z)\mapsto g(qz)A(z)g(z)^{-1}.
\end{equation}
giving $A$ the structure of $(SL(r+1),q)$-connection. 

The Definition \ref{qopflagD} can be reformulated in a local form -- given a section $s(z)$ of $\mathcal{L}_{r+1}$, the oper condition is as follows. Consider the following determinants:
\begin{align}\label{altqW} 
&{W}_i(s(z))=\cr
& \left(s(z)\wedge A(z) s(M_{q}z)\wedge A(M_{q}z) A(z)s(M_{q}^2z)\wedge\dots\wedge \Big(\prod_{j=0}^{i-2}(A(M_{q}^{i-2-j}z)\Big)s(M_{q}^{i-1}z)\right)\Bigg|_{\Lambda^i\cL_{r-i+2}^{q^{i-1}}}\cr
& \quad i=2,\dots, r+1
\end{align}
The oper conditions are equivalent to the fact that (\ref{altqW}) are nonzero.  
\begin{Def} 
We say that $(SL(r+1),q)$-oper has regular singularities defined by the collection of polynomials $\{\Lambda_i(z)\}_{i=1,\dots, r}$ when $\bar{\mathcal{A}}_i$ is an isomorphism away from the zeros of $\Lambda_i(z)$ for $i=1,\dots, r$.  
\end{Def}
In local terms, the regular singularities condition can be reformulated as follows: 
\begin{align}
\label{eq:WPDefs}
W_k(s(z))&=P_1(z) \cdot P_2(M_{q}z)\cdots P_{k}(M_{q}^{k-1}z),  \cr
P_i(z)&=\Lambda_{r}(z)\Lambda_{r-1}(z)\cdots\Lambda_{r-i+1}(z)\,.
\end{align} 

\begin{Def}
The $(SL(r+1),q)$-oper is called $Z-twisted$ if there exists $g(z)\in SL(r+1)(z)$ such that
\begin{eqnarray}    \label{Ag}
A(z)=g(qz)Z g(z)^{-1},
\end{eqnarray}
where $Z$ is a diagonal element of $SL(r+1)$.
\end{Def}

\begin{Def}  
A Miura $(SL(r+1),q)$-oper is a quadruple $(\mathcal{A},E, \mathcal{L}_{\bullet}, \hcL_\bullet)$, where 
$(\mathcal{A},E, \mathcal{L}_{\bullet})$ triple is $(SL(r+1),q)$-oper and the complete flag $\hcL_\bullet$ of subbundles in $E$ is preserved by the q-connection $A$.  
\end{Def}

A natural question is of course how many there are Miura q-opers for a given $Z$-twisted $(SL(r+1),q)$-oper.  

\begin{Prop}
Let $S_{r+1}$ be the symmetric group of $r+1$ elements. There are exactly $(r+1)!$ Miura opers for a given $Z$-twisted $(SL(r+1),q)$-oper if $Z$ is regular semisimple. 
\end{Prop}

Trivializing the flag of bundles $\hat{\mathcal{L}}_\bullet$ and choosing a standard basis $e_1,e_2,\dots, e_{r+1}$ in the space of corresponding sections, we can express relative position between two flags $\mathcal{L}_\bullet$ and $\hat{\mathcal{L}}_\bullet$ using the following determinant:
\begin{equation}\label{qD}
\mathcal{D}_k(s)=e_1\wedge\dots\wedge{e_{r+1-k}}\wedge
s(z)\wedge Z\, s(M_qz)\wedge\dots\wedge Z^{k-1} s(M_q^{k-1}z)\,
\end{equation}
for $k=2,\dots_r+1$ and where $s(z)$ is vector of polynomials with components $\{s_i(z)\}_{i=1,\dots, r+1}$ in the basis of $e_i$'s. 
The functions $\mathcal{D}_k(s)$ has a subset of zeroes, which coincide with those of
$W_k(s)(z)$. The remaining zeros of $\mathcal{D}_k(s)$ are given by points at which the two flags fail to be in general position.
More explicitly, in matrix notation we have
\begin{equation}
\det\begin{pmatrix} \,     1 & \dots & 0 & s_{1}(z) & \xi_1\, s_{1}(M_qz) & \cdots  &  (\xi_1)^{k-1}s_{1}(M_q^{k-1}z) \\ 
 \vdots & \ddots & \vdots& \vdots & \vdots & \ddots & \vdots \\  
0 & \dots & 1&s_{k}(z) &\xi_{k}\, s_{k}(M_qz) & \dots &  (\xi_{k})^{k-1} s_{k}(M_q^{k-1}z)  \\  
0 & \dots & 1&s_{k+1}(z) &\xi_{k+1}\, s_{k+1}(M_qz) & \dots &  (\xi_{k+1}p)^{k-1} s_{k+1}(M_q^{k-1}z)  \\  
\vdots & \ddots & \vdots&\vdots & \vdots & \ddots & \vdots \\
0 & \dots & 0&s_{r+1}(z) & \xi_{r+1} \,s_{r+1}(M_qz) & \dots & (\xi_{r+1})^{k-1}s_{r+1}(M_q^{k-1}z)  \, \end{pmatrix} =\alpha_{k} W_{k}
\cV_{k} \,; 
\label{eq:MiuraQOperCond}
\end{equation}
Here  $Z={\rm diag}(\xi_1, \dots, \xi_{r+1})$ and 
\begin{equation} 
\mathcal{V}_k(z) = \prod_{a=1}^{r_k}(z-v_{k,a})\,,
\label{eq:BaxterRho}
\end{equation}
while $\alpha_k$ are some constants.

Since $\cD_{r+1}(s)=\cW_{r+1}(s)$, we have
$\cV_{r+1}=1$.  We also set $\cV_0=1$; this is consistent with the fact
that \eqref{qD} also makes sense for $k=0$, giving
$\cD_0=e_1\wedge\dots\wedge e_{r+1}$.

We can also rewrite \eqref{eq:MiuraQOperCond} as
\begin{equation}
  \underset{i,j}{\det} \left[\xi_i^{j-1}s^{(j-1)}_{r+1-k+i}(z)\right] = \beta_{k} W_{k} \mathcal{V}_{k}\,,
\label{eq:MiuraDetForm}
\end{equation}
where $i,j = 1,\dots,k$ and $s_l^{(m)}=s_l(M_q^mz)$. 

So far, we could freely exchange $M_q$ and $M_{\epsilon}$ and all of the above definitions could be used to interchange $(SL(r+1),\epsilon)$ and $(SL(r+1),q)$-opers. Now we will encounter the differences.

\subsection{$(SL(r+1),q)$-opers, the $QQ$-system and Bethe Ansatz}

The following theorem gives the relations on the roots of $\mathcal{V}_k(z)$, which as we will see later, will allow to relate $(G,q)$-opers to Bethe Ansatz equations. See \cite{J.R.Li:2012aa} for earlier developments.

\begin{Thm}[\cite{KSZ,koroteev_zeitlin_2023}]\label{qWthKSZ}
Polynomials $\{\mathcal{V}_k(z)\}_{k=1,\dots, r}$ provide the solution to the $QQ$-system 
\begin{equation}
 \xi_{i+1} Q^+_i(M_qz) Q^{-}_i(z) -  \xi_{i} Q^+_i(z)Q^-_i(M_qz) = (\xi_{i+1}-\xi_i)\Lambda_i(z)Q^+_{i-1}(M_q z)Q^+_{i+1}(z)\, ,
\label{eq:dQQrelations}
\end{equation}
so that $Q^+_j(z)=\mathcal{V}_j(z)$. The polynomials $Q^+_j, Q^-_j$ for $j=1,\dots,r$ can be presented using minors
\begin{equation}
Q^+_j(z)= \frac{1}{F_i(z)}\frac{\text{det}\Big( M_{1,\ldots,j} \Big)}{\text{det}\Big( V_{1,\ldots,j} \Big)}\,,
\qquad
Q^-_j(z)= \frac{1}{F_i(z)}\frac{\text{det}\Big( M_{1,\ldots,j-1,j+1} \Big)}{\text{det}\Big( V_{1,\ldots,j-1,j+1} \Big)}\,,
\label{eq:QPolyM}
\end{equation}
where $F_i(z)=M_q^{i-r}W_{r-i}(z)$, 
\begin{equation}
M_{i_1,\ldots,i_j} = \begin{bmatrix} \,  s_{i_1} & \xi_{i_1} s_{i_1}^{(1)} & \cdots & \xi_{i_1}^{j-1} s_{i_1}^{(j-1)} \\ \vdots & \vdots & \ddots & \vdots \\  s_{i_j}  & \xi_{i_j} s_{i_j}^{(1)} & \cdots & \xi_{i_j}^{j-1} s_{i_j}^{(j-1)} \, \end{bmatrix}\,,
\qquad
V_{i_1,\ldots,i_j} =\begin{bmatrix} \,  1 & q\xi_{i_1}  & \cdots & q^{j-1}\xi_{i_1}^{j-1} \\ \vdots & \vdots & \ddots & \vdots \\  1 & q\xi_{i_j} & \cdots &q^{j-1} \xi_{i_j}^{j-1}\, \end{bmatrix}\,,
\label{eq:MM0Ind2}
\end{equation}
are the quantum Wronskian and the Vandermonde matrix respectively \footnote{We rescaled $V$ by powers of $q$ compared to \cite{KSZ} for further convenience.}.
\end{Thm}

Moreover, in the case of semisimple $Z$, for a given $Z$-twisted $(SL(r+1,q))$-oper, the 
$QQ$-system description of the set of Miura opers merely corresponds to the application of symmetric group to the set of $\xi_i$ and  $s_i(z)$ in the context of Theorem (\ref{qWthKSZ}). The system of equations which is a union of $QQ$-systems for all the Miura $(SL(r+1,q))$-opers for a given $Z$-twisted $(SL(r+1),q)$-oper, is known as a {\it full $QQ$-system}, which corresponds to the relations between various minors in the $\mathcal{D}_{r+1}(z)$. 

There is a way to see an algebraic relation between the roots of $Q_i^+(z)$ explicitly. 
Let $\Lambda_j(z)=\prod_{c=1}^{M_j}(z-a_{j,c})$ and $Q_j^+(z)=\prod_{c=1}^{N_j}(z-s_{j,c})$. 
Let us impose the following nondegeneracy condition on the roots of $Q_j^+(z)$. First of all, let us call 
$u,v\in \mathbb{P}^1$ to be $q$-distinct if $q^{\mathbb{Z}}u\cap q^{\mathbb{Z}}v=\varnothing$. Then if 
zeroes of $Q^+_i(z)$, $Q^+_{i\pm 1}$ are $q$-distinct from each other and if 
$\xi, \xi_{i+1}$ are $q$-distinct, we call such $QQ$-system and the corresponding $Z$-twisted Miura $(SL(r+1),q)$-oper {\it nondegenerate}.

Then we have the following theorem.

\begin{Thm}[\cite{Frenkel:2020}]
\label{Th:XXZBethecorr}
The solutions of the nondegenerate $SL(r+1)$ $QQ$-system \eqref{eq:dQQrelations} are in one-to-one correspondence with the solutions to the following algebraic equations between the roots of $\{Q_i^+(z)\}_{i=1,\dots, r}$ known as the Bethe Ansatz equations for $\mathfrak{sl}(n+1)$ XXZ spin chain:
\begin{equation}    \label{eq:bethe}
\frac{Q^+_{i}(qs_{i,k})}{Q^+_{i}(q^{-1}s_{i,k})} \frac{\xi_i}{\xi_{i+1}}=
- \frac{\Lambda_i(s_{i,k}) Q^{+}_{i+1}(qs_{i,k})Q^{+}_{i-1}(s_{i,k})}{\Lambda_i(q^{-1}s_{i,k})Q^{+}_{i+1}(s_{i,k})Q^{+}_{i-1}(q^{-1}s_{i,k})},
\end{equation}
where $i=1,\ldots,r; k=1,\ldots,r+1_i$.
\end{Thm}

Following \cite{koroteev_zeitlin_2023,KSZ} and \cite{KoroteevZeitlinCrelle} we obtain the following theorem.

\begin{Thm}
There is a one-to-one correspondence between solutions of the $QQ$-system (\ref{eq:dQQrelations}) and $Z$-twisted Miura $(SL(r+1),q)$-opers with regular singularities.
\end{Thm}

\subsection{The tRS model and $(SL(N),q)$-opers}

Now let us choose an interesting specific choice of regular singularities and some nondegeneracy conditions.
\begin{Def}\label{canqop}
We will call $Z$-twisted Miura $(SL(r+1,q))$-oper \textit{canonical} if it satisfies the following conditions:
\begin{enumerate}
\item $Z$ is regular semisimple,
\item ${\rm deg}(\mathcal{D}_k)=k$,
\item This oper does not have regular singularities except for the roots of 
\begin{equation}\label{Wrel}
\Lambda(z)=\mathcal{D}_{r+1}(z), 
\end{equation}
which are distinct.
\end{enumerate}
\end{Def}

Then we immediately have the following proposition.

\begin{Prop}
The polynomials $\{s_i(z)\}_{i=1,\dots, r+1}$ describing the line bundle $\mathcal{L}_r$ are 
all of degree one.
\end{Prop}

Without loss of generality we can assume $s_i(z)$ to be monic, namely $s_i(z)=z-p_i$ for some complex $p_i$. The only relation, which determines the space of such objects is the relation (\ref{Wrel}). We will call the space of canonical $Z$-twisted Miura $(SL(r+1,q))$-opers as $q{\rm Op}_Z^{\Lambda}$. We can introduce the space
\begin{eqnarray}
{\rm Fun}\big(q{\rm Op}_Z^{\Lambda}\big)=\frac{\mathbb{C}(q, \xi_i, p_i, a_i)}{\rm Wr},
\end{eqnarray}
where ${\rm Wr}$ stands for the relation (\ref{Wrel}). 

Using Theorem \ref{Th:XXZBethecorr} we have the following statement.

\begin{Prop}
There is an isomorphism of algebras
\begin{equation}
{\rm Fun}\big(q{\rm Op}_Z^{\Lambda}\big)=\frac{\mathbb{C}(q, \xi_i, s_{k,l}, a_i)}{\rm Bethe},
\end{equation}
where {\rm Bethe} stands for the relations (\ref{eq:bethe}), specialized to $q{\rm Op}_Z^{\Lambda}$.
\end{Prop}

Let us reformulate this algebra in a more interesting way. 
Let $T$ be the tRS Lax matrix 
\begin{eqnarray}
T_{ij}=  \frac{\prod\limits_{m \neq j}^{r+1} \left(q^{-1} \xi_i \,  -  \xi_m \,  \right) }{\prod\limits_{l\neq j}^{r+1}(\xi_j-\xi_l)}   p_i\,, \qquad i,j=1,\dots,r+1\,.
\end{eqnarray}
Let the tRS Hamiltonians $ H^{tRS}_k$ be the coefficients of its characteristic polynomial 
\begin{equation}\label{eq:tRSW3}
\textrm{det}\Big(z - T \Big)=\sum_k H^{tRS}_k(q, \{\xi_i\}, \{p_i\})z^k
\end{equation}
Then we can prove the following

\begin{Thm}[\cite{KSZ}]\label{Th:XXz/tRS}
There is an isomorphism of algebras 
\begin{equation}
{\rm Fun}\big(q{\rm Op}_Z^{\Lambda}\big)\cong \frac{\mathbb{C}(q, \xi_i, p_i, a_i)}
{\left( H^{tRS}_k=e_k(a_1, \dots, a_{r+1})\right)_{k=1,\dots, r+1}},
\end{equation}
where $a_1, \dots, a_{r+1}$ are roots of polynomial $\Lambda(z)$ and $e_k$ are elementary symmetric functions of their variables.
\end{Thm}

\begin{proof}
Using Theorem \ref{qWthKSZ} we can put $j=r+1$ in \eqref{eq:QPolyM}
\begin{equation}
P(z)= \frac{\text{det}\Big( M_{1,\ldots,r+1} \Big)}{\text{det}\Big( V_{1,\ldots,r+1} \Big)}\,.
\end{equation}
Notice that
\begin{equation}
M_{1,\ldots,r+1}(z) = V_{1,\ldots,r+1} \cdot z + M_{1,\ldots,r+1}(0)\,.
\end{equation}
We can now simplify the formulae by inverting Vandermonde matrix $V_{1,\ldots,r+1}$ as follows
\begin{equation}
P(z)=\text{det}\left(z-T\right)\,, 
\label{eq:BaxterFlavorPol}
\end{equation}
where 
\begin{equation}
T=-M_{1,\ldots,r+1}(0)\cdot\Big( V_{1,\ldots,r+1} \Big)^{-1}\,.
\label{eq:tRSLaxMatrix}
\end{equation}
A straightforward computation leads us to the desired result. Indeed, the inverse of the Vandermonde matrix reads
\begin{equation}
(V_{1,\dots, n}^{-1})_{t,j} = (-1)^{t+j}q^{-t+1}\frac{S_{n-t,j}(\xi_1,\dots,\xi_{r+1})}{\prod\limits_{l\neq j}^{r+1}(\xi_j-\xi_l)}\,,
\end{equation}
where 
$$
S_{k,j}(\xi_1,\dots,\xi_{r+1})= e_k(\xi_1,\dots,\xi_{j-1},\xi_{j+1},\xi_{r+1})\,,
$$
and 
$$
e_k(\xi_1,\dots,\xi_{r+1})=\sum\limits_{1\leq i_1\leq \dots \leq i_k \leq n}^{r+1} \xi_{i_1}\cdots\xi_{i_k}\,,
$$
are the elementary symmetric polynomials.
Then we have 
$$
\left(-M'_{1,\ldots,r+1}(0)\right)_{ij} = \xi_i^{j-1}p_i\,.
$$
Thus, according to \eqref{eq:tRSLaxMatrix}
\begin{align}
T_{ij} &= \sum_{t=1}^{r+1}\xi_i^{t-1}p_i\cdot (V_{1,\dots,r+1})_{tj}\cr
&= \sum_{t=1}^{r+1}  \frac{(-1)^{t+j}q^{-t+1}\xi_i^{t-1}S_{n-t,j}(\xi_1,\dots,\xi_{r+1})}{\prod\limits_{l\neq j}^{r+1}(\xi_j-\xi_l)} p_i =  \frac{\prod\limits_{m \neq j}^{r+1} \left(q^{-1} \xi_i \,  -  \xi_m \,  \right) }{\prod\limits_{l\neq j}^{r+1}(\xi_j-\xi_l)}   p_i\,,
\end{align}
which is the Lax matrix $\{T_{ij}\}$.
\end{proof}

The functions $H^{tRS}_k$ are known as the tRS Hamiltonians. 

\subsubsection{Quantum/Classical Duality}
Consider the tRS phase space with symplectic form $$\Omega=\sum^{r+1}_{i=1}\frac{dx^i}{x^i}\wedge \frac{dp_i}{p_i}.$$ 
The Hamiltonians $H_k(q, x^i, p_i)$ are known to be mutually commuting with respect to the Poisson bracket corresponding to $\Omega$.
Theorem \ref{Th:XXz/tRS} implies the following.

\begin{Cor}(Trigonometric q-difference Quantum/Classical duality)\label{tqqc}
We have the following isomorphisms:
\begin{eqnarray}
\frac{\mathbb{C}(q, \xi_i, s_{k,l}, a_i)}{\rm Bethe}\cong {\rm Fun}\big(q{\rm Op}_Z^{\Lambda}\big)\cong \frac{\mathbb{C}(q, \xi_i, p_i, a_i)}
{\{ H^{tRS}_k=e_k(a_1, \dots, a_{r+1})\}_{k=1,\dots, r+1}}\,.
\end{eqnarray}
The latter space is isomorphic to the space of functions on the intersection of $\mathscr{L}_1\cap \mathscr{L}_2$ of two Lagrangian subvarieties with respect to form $\Omega$.
\begin{eqnarray} 
\mathscr{L}_1=\{x^i=\xi_i\}_{i=1,\dots, r+1}, \quad \mathscr{L}_2=\{H^{tRS}_i=e_i(\{a_i\})\}_{i=1,\dots, r+1}.
\end{eqnarray}
\end{Cor}

Thus there are two equivalent descriptions of the space of trigonometrically $Z$-twisted q-opers -- in terms of quantum XXZ Bethe Ansatz equations and in terms of energy relations of the classical trigonometric Ruijsenaars-Schneider model.

\subsection{$(SL(r+1),\epsilon)$-opers and tCM Model}
So far we considered $(SL(r+1),q)$-opers. Let us now modify our constructions to $(SL(r+1),\epsilon)$-case.
The general structure and definitions are very similar: we will just give the definitions and propositions which are different.
We start from the correspondence between $(G,\epsilon)$-opers and the additive analogue of a $QQ$-system which we now refer to as the $qQ$-system.
\begin{Thm}[\cite{KSZ}]\label{qWthKSZad}
Polynomials $\{\mathcal{V}_k(z)\}_{k=1,\dots, r}$ provide the solution to the $qQ$-system 
\begin{equation}
 \xi_{i+1} Q^+_i(z+\epsilon) Q^{-}_i(z) -  \xi_{i} Q^+_i(z)Q^-_i(z+\epsilon) = (\xi_{i+1}-\xi_i)\Lambda_i(z)Q_{i-1}(z+\epsilon)Q_{i+1}(z)\, ,
\label{eq:dQQrelationsad}
\end{equation}
so that $Q^+_j(z)=\mathcal{V}_j(z)$ under certain nondegeneracy conditions.
The polynomials $Q^+_j, Q^-_j$ for $j=1,\dots,r+1$ can be presented using quantum Wronskians
\begin{equation}
Q^+_j(z)= \frac{1}{F_i(z)}\frac{\text{det}\Big( M_{1,\ldots,j} \Big)}{\text{det}\Big( V_{1,\ldots,j} \Big)}\,,
\qquad
Q^-_j(z)= \frac{1}{F_i(z)}\frac{\text{det}\Big( M_{1,\ldots,j-1,j+1} \Big)}{\text{det}\Big( V_{1,\ldots,j-1,j+1} \Big)}\,,
\label{eq:QPolyMM}
\end{equation}
where $F_i(z)=W_{r-i}(z+(i-r)\epsilon)$ and
\begin{equation}
M_{i_1,\ldots,i_j}(z) = \begin{bmatrix} \,  s_{i_1}(z) & \xi_{i_1} s_{i_1}(z+\epsilon) & \cdots & \xi_{i_1}^{j-1} s_{i_1}(z+\epsilon(j-1)) \\ \vdots & \vdots & \ddots & \vdots \\  s_{i_j}(z)  & \xi_{i_j} s_{i_j}(z+\epsilon) & \cdots & \xi_{i_j}^{j-1} s_{i_j}(z+\epsilon(j-1)) \, \end{bmatrix}\,,
\quad
V_{i_1,\ldots,i_j} =\begin{bmatrix} \,  1 & \xi_{i_1}  & \cdots & \xi_{i_1}^{j-1} \\ \vdots & \vdots & \ddots & \vdots \\  1 & \xi_{i_j} & \cdots & \xi_{i_j}^{j-1}\, \end{bmatrix}\,.
\label{eq:MM0Indad}
\end{equation}
\end{Thm}

Let us then discuss the analog of Bethe Ansatz equations in this case. First, let us discuss the nondegeneracy relations on the roots of $Q_j^+(z)$. First of all, let us call 
$u,v\in \mathbb{P}^1$ to be $\epsilon$-distinct if 
$u+\mathbb{Z}\epsilon\cap v+\mathbb{Z}\epsilon=\varnothing$. Then if 
zeroes of $Q^+_i(z)$, $Q^+_{i\pm 1}$ are $\epsilon$-distinct from each other and if 
$\xi\neq \xi_{i+1}$, we call such $qQ$-system and the corresponding $Z$-twisted Miura $(SL(r+1),\epsilon)$-oper {\it nondegenerate}.

Then we have the following theorem.

\begin{Thm}[\cite{Frenkel:2020}]
\label{Th:XXZBethecorradd}
The solutions of the nondegenerate $SL(r+1)$ $qQ$-system \eqref{eq:dQQrelationsad} are in one-to-one correspondence with the solutions to the following algebraic equations between the roots of $\{Q_i^+(z)\}_{i=1,\dots, r}$ known as the Bethe Ansatz equations for $\mathfrak{sl}(n+1)$ XXX spin chain:
\begin{equation}    \label{eq:betheadd}
\frac{Q^+_{i}(s_{i,k}+\epsilon)}{Q^+_{i}(s_{i,k}-\epsilon)} \frac{\xi_i}{\xi_{i+1}}=
- \frac{\Lambda_i(s_{i,k}) \cdot Q^{+}_{i+1}(s_{i,k}+\epsilon)\cdot Q^{+}_{i-1}(s_{i,k})}{\Lambda_i(s_{i,k}-\epsilon)\cdot Q^{+}_{i+1}(s_{i,k})\cdot Q^{+}_{i-1}(s_{i,k}-\epsilon)},
\end{equation}
where $i=1,\ldots,r; k=1,\ldots,r+1_i$ and $\Lambda_j(z)=\prod_{c=1}^{M_j}(z-a_{j,c})$ and $Q_j^+(z)=\prod_{c=1}^{N_j}(z-s_{j,c})$.
\end{Thm}

We define the space $\epsilon {\rm Op}^{\Lambda}_Z$ following Definition \ref{canqop}, replacing $M_q$ by $M_{\epsilon}$ so that 
\begin{eqnarray}
{\rm Fun } (\epsilon {\rm Op}^{\Lambda}_Z)=\frac{\mathbb{C}(\epsilon, p_i, \xi_i, a_i)}{\rm Wr}
\end{eqnarray}
where the relations ${\rm Wr}$ are as follows: 
\begin{equation}
  \underset{i,j}{\det} \left[\xi_i^{j-1}s_{r+1-k+i}(z+\epsilon (j-1))\right] = \Lambda(z)\,,
\label{eq:MiuraDetFormxxx}
\end{equation}
where $i,j=1, \dots, r+1$ and $s_i(z)=z-p_i$.

\begin{Thm}
There is a one-to-one correspondence between solutions of the $qQ$-system (\ref{eq:dQQrelationsad}) and $Z$-twisted Miura $(SL(r+1),\epsilon)$-opers with regular singularities.
\end{Thm}

Difference operators corresponding to Miura $(SL(r+1),\epsilon)$-opers in the absence of twist variables appeared in \cite{Mukhin_2005} and they were referred to as discrete Miura opers. Also, see \cite{MR2409414} (Theorem 6.2) where a similar statement was proven where opers were treated as scalar difference operators.

\vskip.1in

Now let us find the connection between $\epsilon{\rm Op}^Z_{\Lambda}$ with the many-body integrable model:  the trigonometric Calogero-Moser system. 

\begin{Thm}\label{Th:XXX/tCM}
There is an isomorphism of algebras 
\begin{equation}
{\rm Fun}\big(\epsilon{\rm Op}_Z^{\Lambda}\big)\cong \frac{\mathbb{C}(\epsilon, \xi_i, p_i, a_i)}
{( H^{tCM}_k=e_k(a_1, \dots, a_{r+1}))_{k=1,\dots, r+1}},
\end{equation}
where 
\begin{equation}\label{eq:tRSW}
\text{det}\Big(z - m \Big)=\sum_k H^{tCM}_k(\epsilon, \{\xi_i\}, \{p_i\})z^k
\end{equation}
such that  
$\Lambda(z)=\prod_{i=1}^{r+1}(z-a_i)=\sum_kz^ke_k(a_1, \dots, a_{r+1})$ 
 and $m=\{m_{ij}\}_{i,j=1,\dots, r+1}$ is the tCM Lax matrix:
\begin{eqnarray}\label{eq:tCMLaxm}
&& m_{ii}=p_i - \epsilon\xi_i\sum_{k\neq i}\frac{1}{\xi_i-\xi_k}, \quad i=1,\dots, r+1,\nonumber\\
&& m_{ij}=\frac{\epsilon\xi_i}{\xi_i-\xi_j}\frac{\prod\limits_{k\neq i}(\xi_i-\xi_k)}{\prod\limits_{k\neq j}(\xi_j-\xi_k)}\,\quad i,j=1,\dots, r+1, ~i\neq j.
\end{eqnarray}

\end{Thm}

\begin{proof}
Using Theorem \ref{qWthKSZ} we can put $j=r+1$ and write
\begin{equation}
P(z)= \frac{\text{det}\Big( M_{1,\ldots,r+1} \Big)(z)}{\text{det}\Big( V_{1,\ldots,r+1} \Big)}\,,
\end{equation}
where the Vandermonde matrix reads $\left(V_{1,\ldots,r+1}\right)_{ij}=\xi_i^{j-1}$. Denote $q_i(z)=z-p_i$ then 
$\left(M_{1,\ldots,r+1}\right)_{i,j}=\xi_i^{j-1}(z-p_i+(j-1)\epsilon)$, and we can rewrite the above equation as
\begin{equation}\label{eq:LaxCalc}
P(z)= \text{det}\Big(z + M_{1,\dots,r+1}(0)V^{-1}_{1,\dots, r+1}\Big)\,.
\end{equation}
A short calculation shows that 
\begin{equation}
-M_{1,\dots,r+1}(0)V^{-1}_{1,\dots, r+1}= m\,,
\end{equation}
is precisely the Lax matrix of the trigonometric Calogero-Moser (tCM) system \eqref{eq:tCMLaxm}.
Indeed, $-M_{1,\dots,r+1}(0)_{ij}=\xi_i^{j-1} p_i-(j-1)\xi_i^{j-1}\epsilon$. The inverse of the Vandermonde matrix \eqref{eq:MM0Indad} reads
\begin{equation}
(V_{1,\dots, r+1}^{-1})_{tj} = (-1)^{t+j}\frac{S_{r+1-t,j}(\xi_1,\dots,\xi_{r+1})}{\prod\limits_{l\neq j}^{r+1}(\xi_j-\xi_l)}\,,
\end{equation}
Then 
\begin{equation}
m_{ij}=p_i\delta_{ij}-\epsilon\displaystyle\sum_{t=1}^{r+1} \frac{(-1)^{t+j} (t-1)\xi_i^{t-1}S_{r+1-t,j}(\xi_1,\dots,\xi_{r+1})}{\prod\limits_{l\neq j}^{r+1}(\xi_j-\xi_l)}
\end{equation}
In order to understand the sum in the numerator of the second term above consider the following identity
\begin{equation}
\prod_{i=1}^{r+1}(u-\xi_i) = \sum_{k=0}^{r+1} (-1)^{r+1-k} u^{k} e_{r+1-k}(\xi_1,\dots,\xi_{r+1})\,,
\end{equation}
also, if we decide to remove $\xi_j$ from the set of variables, this identity holds
\begin{equation}
\prod_{\substack{i=1\\ i\neq j}}^{r+1}(u-\xi_i) = \sum_{k=0}^{r+1} (-1)^{r+1-k} u^{k} S_{r+1-k,j}(\xi_1,\dots,\xi_{r+1})\,.
\end{equation}
Now let us differentiate both parts of the above equation with respect to $u$
\begin{align}
\sum_{l=1}^{r+1}\prod_{\substack{m=1\\ m\neq l,j}}^{r+1}(u-\xi_m) &= \sum_{k=0}^{r} (-1)^{r+1-k} k \, u^{k-1} S_{r+1-k,j}(\xi_1,\dots,\xi_{r+1})\cr
&=u^{-1}\sum_{k=1}^{r} (-1)^{r-k} (k-1) \, u^{k-1} S_{r+1-k,j}(\xi_1,\dots,\xi_{r+1})\,.
\end{align}
Then, after substituting $u=\xi_i$ we get
\begin{equation}
\sum_{k=1}^{r} (-1)^{r-k} (k-1) \, \xi_i^{k-1} S_{r+1-k,j}(\xi_1,\dots,\xi_{r+1})=\xi_i \sum_{l=1}^{r+1}\prod_{\substack{m=1\\ m\neq l,j}}^{r+1}(\xi_i-\xi_m)\,.
\end{equation}
Now if $i=j$ then the expression in the right hand side contains $r$ terms and reads $\xi_i \sum_{l=1}^{r+1}\prod_{\substack{m=1\\ m\neq l,i}}^{r+1}(\xi_i-\xi_m)$. Whereas, for $i\neq j$, only one term is nonzero $\xi_i \prod_{\substack{m=1\\ m\neq i,j}}^{r+1}(\xi_i-\xi_m)$.
Combining the results together we obtain \eqref{eq:tCMLaxm}.
\end{proof}

The functions $H_k$ are known as the $tCM$ Hamiltonians. 

Let us reformulate now the statement about the quantum/classical duality in this case.
The functions $H_k$ are known as the $tCM$ Hamiltonians. Consider the tRS phase space with symplectic form
$$
\Omega=\sum^{r+1}_{i=1}\frac{dx^i}{x^i}\wedge \frac{dp_i}{p_i}.
$$ $H^{tCM}_k(\epsilon, x^i, p_i)$ are known to be mutually commuting with respect to the Poisson bracket corresponding to $\Omega$.
Let us reformulate now the statement about the quantum/classical duality in this case, which is the analog of Corollary \ref{tqqc}. To do that let us combine Theorems \ref{Th:XXZBethecorr}, \ref{Th:XXX/tCM}.
\begin{Cor}(Trigonometric $\epsilon$-difference Quantum/Classical duality)\label{rteo}
We have the following isomorphisms:
\begin{eqnarray}
\frac{\mathbb{C}(\epsilon, \xi_i, s_{k,l}, a_i)}{\rm Bethe}\cong {\rm Fun}\big(\epsilon{\rm Op}_Z^{\Lambda}\big)\cong \frac{\mathbb{C}(\epsilon, \xi_i, p_i, a_i)}
{\{ H^{tCM}_k=e_k(a_1, \dots, a_{r+1})\}_{k=1,\dots, r+1}},
\end{eqnarray}
where ${\rm Bethe}$ stands for equations \eqref{eq:betheadd}. The latter space is isomorphic to the space of functions on the intersection $\mathscr{L}_1\cap \mathscr{L}_2$ of two Lagrangian subvarieties:
\begin{eqnarray} 
\mathscr{L}_1=\{x^i=\xi_i\}_{i=1,\dots, r+1}, \quad \mathscr{L}_2=\{H^{tCM}_i=e_i(\{a_i\})\}_{i=1,\dots, r+1}.
\end{eqnarray}
\end{Cor}

\section{$SL(r+1)$-Opers with Regular Singularities: Rational and Trigonometric Z-Twists}\label{Sec:DiffOpers}
Let us review the definition of $SL(r+1)$-opers (see \cite{KSZ}). 
In this paper, we give a complete local description, which was omitted in \cite{KSZ}.

\begin{Def} A \emph{$\GL(r+1)$-oper} on $\mathbb{P}^1$ is a triple
  $(E,\nabla,\cL_\bullet)$, where $E$ is a rank $r+1$ vector bundle on
  $\mathbb{P}^1$, $\nabla:E\to E\otimes K$ is a connection, and
  $\cL_\bullet$ is a complete flag of subbundles such that $\nabla$
  maps $\cL_i$ into $\cL_{i+1}\otimes K$ and the induced maps
  $\bar{\nabla}_i:\cL_i/\cL_{i+1}\to \cL_{i-1}/\cL_{i}\otimes K$ are
  isomorphisms for $i=1,\dots,r$. The triple is called an
  \emph{$SL(r+1)$-oper} if the structure group of the flat
  $GL(r+1)$-bundle may be reduced to $SL(r+1)$.
\end{Def}

Similarly to the difference case one can describe it using a local section of $\mathcal{L}_{r+1}$
Indeed if $s$ is a section of 
$\cL_{r+1}$, for $i=1,\dots,r+1$, let us denote 
\begin{equation} W_i(s)(z)=\left.\left(s(z)\wedge\nabla_z
    s(z)\wedge\dots\wedge\nabla_z^{i-1} s(z)\right)\right|_{\Lambda^i\cL_{i}}
\end{equation}
for $i=2,\dots,r+1$. 
Then $(E,\nabla,\cL_\bullet)$ is an oper if and only if for each $z$,
there exists a local section of $\cL_{r+1}$ for which $W_i(s)(z)\ne 0$
for all $i$.  

The isomorphism condition needs to be relaxed to allow the oper to have regular singularities.
\begin{Def}
We say that an $SL(r+1)$-oper has regular singularities defined by the collection of polynomials $\{\Lambda_i(z)\}_{i=1,\dots, r}$ when $\bar{\nabla}_i$ is an isomorphism away from zeroes of $\Lambda_i(z)$ for $i=1,\dots, r$.  
\end{Def}
Thus we have 
\begin{eqnarray}
\label{eq:WPDefsdiff}
W_k(s(z))=P_1(z) \cdot P_2(z)\cdots P_{k}(z),  \quad 
P_i(z)=\Lambda_{r}(z)\Lambda_{r-1}(z)\cdots\Lambda_{r-i+1}(z)\,.
\end{eqnarray}  
Let us introduce the following notation for the roots of polynomials of $\Lambda_j$:
\begin{eqnarray}
\Lambda_j(z)=c_j\prod^{M}_{m}(\zeta-a_m)^{l_m^{j}}.
\end{eqnarray}
 where $l_m^{j}$ stands for the multiplicities of roots $a_m$.

\begin{Def}
A Miura oper with regular singularities is a quadruple $(E,\nabla,\cL_\bullet,\hcL_\bullet)$ where
$(E,\nabla,\cL_\bullet)$ is an oper with regular singularities and
$\hcL_\bullet$ is a complete flag of subbundles preserved by $\nabla$.
\end{Def}

Now we give two versions of $Z$-twisted condition.

\begin{Def}
We call $SL(r+1)$-oper {\it rationally $Z$-twisted}, if the corresponding connection is gauge equivalent to the constant diagonal element $Z\in\mathfrak{sl}({r+1})$:
\begin{eqnarray}
\nabla_z=v(z)(\partial_z+Z)v(z)^{-1}, 
\end{eqnarray}
where $v(z)\in\mathfrak{sl}(r+1)(z)$.
\end{Def}

Let us make the coordinate transformation $z=e^{w}$, where $w$ is identified with $w+2\pi i$, i.e. the coordinate on the infinite cylinder.  

\begin{Def}
We call $SL(r+1)$ oper with regular singularities {\it trigonometrically  $Z$-twisted}, if the corresponding connection is gauge equivalent to a simple polar term:
\begin{eqnarray}
\nabla_{z}=v(z)\Big(\partial_z+\frac{Z}{z}\Big)v(z)^{-1}, 
\end{eqnarray}
where $v(z)\in SL(r+1)(z)$ and $Z$ is a Cartan element of $SL(r+1)$.
\end{Def}

Note, that in the trigonometric case, the oper connection has a nontrivial monodromy equal to $e^{2\pi Z}$ around zero on $\mathbb{P}^1$. 
To see that one should pass to the coordinate on a cylinder: $z=e^{\zeta}$, then 
\begin{eqnarray}
\nabla_{\zeta}=v\big(e^\zeta\big)(\partial_\zeta+Z)v\big(e^\zeta\big)^{-1}, 
\end{eqnarray}
That justifies the name `trigonometric' since $Z$-twist is similar to a standard rational case, but we are working over the space of hyperbolic functions of $\zeta$.

In both cases the following theorem is true.

\begin{Prop}[\cite{Brinson:2021ww}]
There are exactly $(r+1)!$ Miura opers for a given $Z$-twisted $SL(r+1)$-oper if $Z$ is regular semisimple. 
\end{Prop}
\begin{Rem}
This theorem was proven in \cite{Brinson:2021ww} for a rational twist. However, it is straightforward to generalize it to the trigonometric case.
\end{Rem}
In the next subsection, we study in detail the rationally $Z$-twisted Miura $SL(r+1)$-opers with regular singularities and after that the trigonometric case.

\subsection{Rationally $Z$-twisted Miura $SL(r+1)$-opers and $qq$-systems.}
Given a Miura oper, choose a trivialization of $E$ on
$\mathbb{P}^1$ such that $\hcL_\bullet$ is generated by the ordered basis $e_1,\dots, e_{r+1}$, while $\nabla_z=\partial_z+Z$, where $Z={\rm diag}(\gamma_1, \dots, \gamma_{r+1})$.
If $s$ is a section generating $\cL_{r+1}$ then for $i=1,\dots,r+1$ we get:
\begin{equation} \mathcal{D}_i(s)(z)=e_1\wedge\dots\wedge e_{r+1-i}\wedge
  s(z)\wedge(\partial_z +Z)s(z)\wedge \dots\wedge(\partial_z+Z)^{i-1} s(z).
\end{equation}

Then rationally $Z$-twisted Miura oper condition can be written as
\begin{eqnarray}
&&\det\begin{pmatrix} \,     1 & \dots & 0 & s_{1}(z) & (\gamma_{1}+\partial_z) s_{1}(z) & \cdots & (\gamma_{1}+\partial_z)^{k-1}s_{1}(z) \\ 
 \vdots & \ddots & \vdots& \vdots & \vdots & \ddots & \vdots \\  
0 & \dots & 1&s_{r+1-k}(z) &(\gamma_{n-k}+\partial_z) s_{r+1-k}(z) &  \dots & (\gamma_{1}+\partial_z)^{k-1} s_{r+1-k}(z)  \\  
0 & \dots & 0&s_{r-k+2}(z) & (\gamma_{n-k+1}+\partial_z)s_{r-k+2}(z) &  \dots & (\gamma_{r-k+2}+\partial_z)^{k-1} s_{r-k+2}(z)  \\
\vdots & \ddots & \vdots&\vdots & \vdots & \ddots & \vdots \\
0 & \dots & 0&s_{r+1}(z) & (\gamma_{N}+\partial_\zeta) s_{r+1}(z) & \cdots & (\gamma_{r+1}+\partial_z)^{k-1}(z)  \, \end{pmatrix} =\nonumber\\
&&\beta_{k} W_{k}(z)\nonumber
\cV_{k}(z) \,,
\label{eq:MiuraqqperCond}
\end{eqnarray}
so that the roots of the polynomials $\mathcal{V}(z)$ correspond to the points when 
flags $\mathcal{L}_{\bullet}$ and  $\hat{\mathcal{L}}_{\bullet}$ are in non-generic position and $W_k(z)$ satisfy
\begin{equation}
  \underset{i,j}{\det} \left[(\gamma_{n-k+i}+\partial_\zeta)^{j-1} s_{n-k+i}(\zeta)\right] = \beta_{k} W_{k}(\zeta)\,,
\label{eq:difMiuraDetFormdqq}
\end{equation}
where $i,j = 1,\dots,k$,

Clearly, we want to repeat the same steps, i.e. find the relations satisfied by suitably normalized polynomials $\mathcal{V}_k$.  
Let us study the minors $\mathcal{D}_i(s)(z)$ in detail. 
We start from the following definition.

\begin{Def}
Let $s_i(z)$ be $n-1$ times differentiable functions, let  $\nabla_i=\partial_\zeta+\gamma_i$.
The $Z$-twisted Wronskian of functions $s_1,\dots s_{r+1}$ is given by
\begin{equation}
W_{r+1}^{\gamma_1,\dots,\gamma_{r+1}}(s_1,\dots s_{r+1})= \begin{pmatrix} \,  s_{1} & \nabla_1 s_{1} & \cdots & \nabla_1^{r} s_{1}\\ \vdots & \vdots & \ddots & \vdots \\  s_{n}  & \nabla_{r+1} s_{r+1} & \cdots & \nabla_{r+1}^{r} s_{r+1} \,. \end{pmatrix}
\end{equation}
\end{Def}

\begin{Rem}
Notice that the twisted Wronskian can be obtained by replacing the derivative $\partial_z\mapsto e^{-Zz}\partial_z e^{Zz}$ in the regular (untwisted) Wronskian $W=\det (\partial_z^{j-1}s_i)$, where matrix $Z$ acts on the column vectors $(1,0\dots,0),\dots,(0,\dots,0,1)$ and with eigenvalues $\gamma_1,\dots,\gamma_{r+1}$. This fact justifies the adjective \textit{twisted} in the name.
\end{Rem}

Let us prove some combinatorial lemmas (see Appendix A of \cite{MV2002} for the untwisted version $\gamma_i=0$).

\begin{Lem}
\label{Th:LemWr1}
Let $s_2,\dots s_n$ be differentiable functions of $z$, then
\begin{equation}
W^{\gamma_1,\gamma_2,\dots,\gamma_{n}}_{n}(e^{-\gamma_{1}z},s_2,\dots s_{n})=e^{-\gamma_{1}z} W^{\gamma_2,\dots,\gamma_{n}}_{n-1}(\nabla_2 s_2,\dots, \nabla_{n} s_{n})
\end{equation}
\end{Lem}

\begin{proof}
Obvious, since $\nabla_{1} e^{-\gamma_{1}z}=(\partial_z+\gamma_{1})e^{-\gamma_{1}z}=0$.
\end{proof}

\begin{Lem}
\label{Th:LemRescaling}
Let $s_1,\dots s_n$ be differentiable functions of $z$, then
\begin{equation}
W^{\gamma_1,\dots,\gamma_{n}}_{n}(fs_1,\dots f s_{n})=f^{n} \cdot W^{\gamma_1,\dots,\gamma_{n}}_{n}(s_1,\dots,s_{n})
\end{equation}
\end{Lem}

\begin{proof}
We proceed by induction. The base case $n=1$ is obvious. Suppose the lemma is true for $n=n_0-1$. Let us compare the following differential equations on function $s_1$
\begin{equation}
W^{\gamma_1,\dots,\gamma_{n_0}}_{n_0}(fs_1,\dots f s_{n_0})=0\,,\qquad \text{vs} \qquad  W^{\gamma_1,\dots,\gamma_{n_0}}_{n_0}(s_1,\dots s_{n_0})=0\,.
\end{equation}
Both have functions $s_2,\dots,s_{n_0}$ as solutions therefore they coincide up to multiplication by a function. One can see that the coefficients of $s_1^{(n_0-1)}$ differ by $f^{n_0}$ by the induction hypothesis. 
\end{proof}

Finally, we will formulate the bilinear relations for twisted Wronskians, which we wanted to find.

\begin{Lem}
\label{Th:LemmaLCid}
\begin{align}
W^{\gamma_{n},\gamma_{n+1}}_{2}\big(W^{\gamma_1,\dots,\gamma_{n}}_{n}(s_1,\dots,s_{n}),& W^{\gamma_1,\dots,\gamma_{n-1},\gamma_{n+1}}_{n}(s_1,\dots,s_{n-1},s_{n+1})\big)\cr
&=W^{\gamma_1,\dots,\gamma_{n-1}}_{n-1}(s_1,\dots,s_{n-1})W^{\gamma_1,\dots,\gamma_{n+1}}_{n+1}(s_1,\dots, s_{n+1})
\end{align}
\end{Lem}

\begin{proof}
Again we shall proceed by induction. The base case $n=1$ is obvious. Suppose the statement holds for $n=n_0-1$. 

Let us divide both sides of the equation for $n=n_0$ by $s_{1}^{2n_0}e^{2{n_0}\gamma_{1}z}$ and use Lemma \ref{Th:LemRescaling} to carry it into the Wronskians. Then one of the functions in each Wronskian will be $e^{-\gamma_{1}z}$. 
Using Lemma \ref{Th:LemWr1} we reduce the equation to the inductive assumption where the Wronskians are calculated for functions $g_i=\nabla_i(e^{-\gamma_1 z}s_i/s_1)$.
\end{proof}

Now, let us introduce the $\mathfrak{sl}_{r+1}$ $qq$-system \cite{MV2002,J.R.Li:2012aa,Brinson:2021ww}. It is the following system of equations
\begin{equation}
\label{eq:QQGoper}
q^+_i(z)\partial_{z} q^-_i(z)-q^-_i(z)\partial_{z} q^+_i(z)+(\gamma_{i+1} -\gamma_{i})q^+_i(z)q^-_i(z)=\Lambda_i(z)q^+_{j-1}(z)q^+_{j+1}(z)
\end{equation}
We can rewrite \eqref{eq:QQGoper} in terms of twisted Wronskians:
\begin{equation}
\label{eq:QQGoperW}
W_2^{\gamma_i,\gamma_{i+1}}(q_i^+,q_i^-)(z)=\Lambda_i(\zeta) q^{i-1}_+(\zeta)q^{i+1}_+(z)\,.
\end{equation}



\begin{Lem}
The system of equations \eqref{eq:QQGoperW} is equivalent to the following set of equations
\begin{equation}
\label{eq:ddtwisted}
W^{i,i+1}_2(d_i^+,d_i^-)(z)=(\gamma_{i+1}-\gamma_i)(d^{i-1}_+\cdot d^{i+1}_+)(z)\,.
\end{equation}
where $d_i^+(z)=F_i q_i^+(\zeta)$ and $d_i^-(z)=\frac{F_i q_i^+(z)}{\gamma_{i+1}-\gamma_i}$, where functions $F_i$ satisfy
\begin{equation}
\label{eq:Lambdacond}
\Lambda_i(\zeta)=\frac{F_{i-1}F_{i+1}}{F_i^2}\,.
\end{equation}
\end{Lem}

\begin{proof}
We can easilty check that the condition \eqref{eq:Lambdacond} is satisfied provided that $F_i (\zeta)= W_{r-i} (\zeta)$.
\end{proof}

\vskip.1in

For $G=SL(r+1)$ the following theorem can be formulated analogously to Theorem \ref{qWthKSZ} and can be proven in a similar way.
\begin{Thm}\label{Wth}
Polynomials $\{\mathcal{V}_k(\zeta)\}_{k=1,\dots, r}$ from \eqref{eq:difMiuraDetFormdqq} give the solution to the $qq$-system \eqref{eq:QQGoper} so that $\mathcal{V}_k(\zeta)=q^+_k(\zeta)$ in the presence of nondegeneracy conditions described above. 
The polynomials $d^+_j,  d^-_j$ for $j=1,\dots,r+1$, which satisfy the $dd$-system \eqref{eq:ddtwisted}, can be presented using twisted Wronskians as follows
\begin{equation}
d^+_j(u)= \frac{W^{1,\dots,j}_j(s_1,\dots,s_j)}{\text{det}\Big( V_{1,\ldots,j} \Big)}\,,
\qquad
 d^-_j(u)= \frac{W^{1,\dots,j-1,j+1}_j(s_1,\dots,s_{j-1},s_{j+1})}{\text{det}\Big( V_{1,\ldots,j-1,j+1} \Big)}\,,
\label{eq:QPolyMqq}
\end{equation}
where
\begin{equation}
V_{i_1,\ldots,i_j} =\begin{pmatrix} \,  1 & \gamma_{i_1}& \cdots & \gamma_{i_1}^{j-1} \\ \vdots & \vdots & \ddots & \vdots \\  1  & \gamma_{i_j} & \cdots & \gamma_{i_j}^{j-1}   \end{pmatrix}\,,
\label{eq:MM0Ind}
\end{equation}
is a Vandermonde matrix.
\end{Thm}

\begin{proof}
The proof follows by direct substitutions of  \eqref{eq:QPolyMqq} into \eqref{eq:ddtwisted} and then applying
Lemma \ref{Th:LemmaLCid}. 
\end{proof}

The algebraic relations between the roots of $q_i^+(z)$ (and therefore, $d_i^\pm(z)$) can be expressed explicitly, once some nondegeneracy conditions are imposed. 
Let $\Lambda_j(z)=\prod_{c=1}^{M_j}(z-a_{j,c})$ and $q_j^+(z)=\prod_{c=1}^{N_j}(z-s_{j,c})$. 
The nondegeneracy condition on the roots of $q_j^+(z)$ is as follows. If 
zeroes of $q^+_i(z)$, $q^+_{i\pm 1}$ are distinct from each other and if $Z$ is regular semisimple, 
we call such $qq$-system and the corresponding $Z$-twisted Miura $SL(r+1)$-oper {\it nondegenerate}.

Then we have the following theorem.

\begin{Thm}
\label{Th:rGBethecorr}
The solutions of the nondegenerate $\mathfrak{sl}(r+1)$ $qq$-system \eqref{eq:QQGoper} are in one-to-one correspondence with the solutions to the following algebraic equations between the roots of $\{q_i^+(z)\}_{i=1,\dots, r}$ known as the Bethe Ansatz equations for rational $\mathfrak{sl}(r+1)$ Gaudin model:
\begin{equation}
\label{eq:rGaudinBethe}
\gamma_{i+1}-\gamma_i-\sum_{(j,b)\neq(i,a)}\frac{a_{ij}}{s_{i,a}-s_{j,b}}+\sum_{c=1}^{M_i}\frac{1}{s_{i,a}-a_{i,c}}=0\,,
\end{equation}
for $i=1,\dots, r+1$, where $a_{ij}$ is the Cartan matrix of $\mathfrak{sl}(r+1)$.
\end{Thm}

In the nondegenerate case, the following theorem is true.

\begin{Thm}[\cite{Brinson:2021ww}]
\label{qqrop}
There is a one-to-one correspondence between the set of nondegenerate $Z$-twisted  Miura $SL(r+1)$ opers and the set of nondegenerate polynomial solutions of the $qq$-system \eqref{eq:QQGoper}.
\end{Thm}

Again, as in the difference case for a semisimple $Z$ and a given $Z$-twisted $SL(r+1)$-oper, 
the $qq$-system description of the set of Miura opers just corresponds to the application of symmetric group to the set of $\gamma_i$ and  $s_i(z)$ in the context of Theorem \ref{Wth}. The system of equations which is a union of $QQ$-systems for all the Miura $(SL(r+1),q)$-opers for a given $Z$-twisted $(SL(r+1),q)$-oper, is known as a {\it full $QQ$-system}, which corresponds to the relations between various minors in the $\mathcal{D}_{r+1}(z)$.

\subsection{Rationally $Z$-twisted $SL(r+1)$-opers and the rCM model}
Now let us define the canonical space of opers as we did for the q-opers
\begin{Def}\label{canqop2}
We will call rationally $Z$-twisted Miura $SL(r+1)$-oper canonical if it satisfies the following conditions:
\begin{enumerate}
\item $Z$ is regular semisimple,
\item ${\rm deg}(\mathcal{D}_k)=k$,
\item This oper does not have regular singularities except for the roots of 
\begin{equation}\label{Wreld}
\Lambda(z)=\mathcal{D}_{r+1}(z), 
\end{equation}
which are distinct.
\end{enumerate}
\end{Def}

Again, it follows that 

\begin{Prop}
Consider a canonical $SL(r+1)$-oper $(A,E,\mathcal{L}_\bullet)$.
Let the line subbundle in the complete flag $\mathcal{L}_\bullet$ be locally described as $\mathcal{L}_r=\text{Span}(s_1,\dots,s_{r+1})(z)$. Then the polynomials $s_i(z)$ are of degree one.
\end{Prop}

Again,without loss of generality we can assume $s_i(z)$ to be monic, i.e. $s(z)=z-p_i$, so that the only relation, which determines the space of such objects is the relation (\ref{Wreld}). We will call the space of canonical $Z$-twisted Miura $SL(r+1)$-opers as $r{\rm Op}_Z^{\Lambda}$. We can introduce the space
\begin{eqnarray}
{\rm Fun}\big(r{\rm Op}_Z^{\Lambda}\big)=\frac{\mathbb{C}(\gamma_i, p_i, a_i)}{\rm Wr},
\end{eqnarray}
where ${\rm Wr}$ stands for the relation (\ref{Wreld}) and $\{a_i\}_{i=1,\dots, r+1}$ are the roots of $\Lambda$. 

Using Theorem \ref{Th:rGBethecorr} we have the following statement.

\begin{Prop}
There is an isomorphism of algebras
\begin{equation}
{\rm Fun}\big(r{\rm Op}_Z^{\Lambda}\big)=\frac{\mathbb{C}(\gamma_i, s_{k,l}, a_i)}{\rm Bethe},
\end{equation}
where {\rm Bethe} stands for the relations (\ref{eq:rGaudinBethe}).
\end{Prop}

Now let us look at the structure of twisted Wronskian, which governs r${\rm Op}_Z^{\Lambda}$
Since all derivatives of $s_i(z)$ starting from the second one are trivial, equation \eqref{Wreld} reads
\begin{equation}
  \underset{i,j}{\det} \left[z\gamma_i^{j-1}-p_i\gamma_i^{j-1}+(j-1)\gamma_i^{j-2}\right] =c\prod^{r+1}_{i=1}(z-a_i)\,,
\label{eq:difMiuraDetFormd}
\end{equation}
where $\Lambda(z)=c\prod^{r+1}_{i=1}(z-a_i)$.

Thus we have the following theorem
\begin{Thm}\label{Th:rGaudin/rCM}

There is an isomorphism of algebras 
\begin{equation}
{\rm Fun}\big(r{\rm Op}_Z^{\Lambda}\big)\cong \frac{\mathbb{C}( \gamma_i, p_i, a_i)}
{\{ H^{rCM}_k=e_k(a_1, \dots, a_{r+1})\}_{k=1,\dots, r}},
\end{equation}
where 
\begin{equation}\label{eq:tRSW2}
\text{det}\Big(z - t \Big)=\sum_k H^{rCM}_k( \{\xi_i\}, \{p_i\})z^k
\end{equation}
where  $t$ is the rCM Lax matrix $\{t_{ij}\}_{i,j=1,\dots, r+1}$:
\begin{eqnarray}\label{eq:rCMoffdiag}
&&t_{ii}=p_i - \sum_{j\neq i}\frac{1}{\gamma_i-\gamma_j}\,,\nonumber\\
&&t_{ij}=\frac{1}{\gamma_i-\gamma_j}\cdot\frac{\prod_{k\neq j}(\gamma_i-\gamma_k)}{\prod_{l\neq i}(\gamma_j-\gamma_l)}\,, ~i,j=1,\dots, r+1; ~ i\neq j.
\end{eqnarray}
\end{Thm}

\begin{proof}
The Lax matrix can be computed analogously to \eqref{eq:LaxCalc}
\begin{equation}
t= -M(0)_{1,\dots,r+1}\cdot V^{-1}_{1,\dots,r+1}\,,
\end{equation}
where $\left(M(0)_{1,\dots,r+1}\right)_{ij}=-p_i\gamma_i^{j-1}+(j-1)\gamma_i^{j-2}$, for $i,j=1,\dots, r$ and $\left(V_{1,\dots,r+1}\right)_{ij}=\gamma_i^{j-1}$ is the Vandermonde matrix. Carrying out the matrix multiplication we obtain 
\begin{equation}
t_{ij}=p_i\delta_{ij}-\displaystyle\sum_{t=1}^{r+1} \frac{(-1)^{t+j} (t-1)\gamma_i^{t-2}S_{r+1-t,j}(\gamma_1,\dots,\gamma_{r+1})}{\prod\limits_{l\neq j}^{r+1}(\gamma_j-\gamma_l)}\,.
\end{equation}
Proceeding by analogy with the proof of Theorem \ref{Th:XXX/tCM} we get the diagonal components 
\begin{equation}
t_{ii}=p_i - \sum_{j\neq i}\frac{1}{\gamma_i-\gamma_j}\,,
\end{equation}
and the off-diagonal components
\begin{equation}
t_{ij}=\frac{1}{\gamma_i-\gamma_j}\cdot\frac{\prod_{k\neq j}(\gamma_i-\gamma_k)}{\prod_{l\neq i}(\gamma_j-\gamma_l)}\,
\end{equation}
of the rational Calogero-Moser Lax matrix.
\end{proof}

Thus we again obtain the relation on the oper space, now of $r{\rm Op}^Z_{\Lambda}$ to many-body system: the 
 functions $H^{rCM}_k$ are known as $rCM$ Hamiltonians. Considering the phase space with symplectic form $\Omega=\sum^{r+1}_{i=1}\frac{dx^i}{x^i}\wedge \frac{dp_i}{p_i},$ one finds that $H^{rCM}_k(x^i, p_i)$ are  mutually commuting with respect to the Poisson bracket corresponding to $\Omega$.
Thus Theorem \ref{Th:rGaudin/rCM} implies the following.

\begin{Cor}(Rational differential Quantum/Classical duality)\label{rdqc}
We have the following isomorphisms:
\begin{eqnarray}
\frac{\mathbb{C}( \gamma_i, s_{k,l}, a_i)}{\rm Bethe}\cong {\rm Fun}\big(r{\rm Op}_Z^{\Lambda}\big)\cong \frac{\mathbb{C}( \gamma_i, p_i, a_i)}
{( H^{rCM}_k=e_k(a_1, \dots, a_{r+1}))_{k=1,\dots, r}},
\end{eqnarray}
where the latter space can be thought of as the space of functions on the intersection $\mathscr{L}_1\cap \mathscr{L}_2$ of two Lagrangian subvarieties:
\begin{eqnarray} 
\mathscr{L}_1=\{x^i=\xi_i\}_{i=1,\dots, r+1}, \quad \mathscr{L}_2=\{H^{tCM}_i=e_i(\{a_i\})\}_{i=1,\dots, r+1}.
\end{eqnarray}
\end{Cor}

\subsection{Trigonometrically twisted $SL(r+1)$-opers, trigonometric Gaudin system and rRS Model}
Let us now characterize what happens in the case of trigonometric twist. For the sake of calculations, and for generalizations of most formulas, one can switch to the variable $\zeta$ on a cylinder. 
In these terms, local sections are parametrized as $s_i(z)=s_i(e^{\zeta})$, while the twisted 
Wronskians have the form:
\begin{equation} \label{Dtrig}
\mathcal{D}_i(s)(z)=e_1\wedge\dots\wedge e_{r+1-i}\wedge
  s(z)\wedge(\partial_\zeta +Z)s(e^{\zeta})\wedge \dots\wedge(\partial_z+Z)^{i-1} s(e^{\zeta}).
\end{equation}   
The related qq$^t$-system is as follows\footnote{The superscript $^t$ stands for `trigonometric' since the dependence of the variables now involved exponents.}:
\begin{equation}
\label{eq:QQGoperttr}
q^+_i(e^\zeta)\partial_{\zeta} q^-_i(e^\zeta)-q^-_i(e^\zeta)\partial_{\zeta} q^+_i(e^\zeta)+(\gamma_i -\gamma_{i+1})q^+_i(e^\zeta)q^-_i(e^\zeta)=\Lambda_i(e^\zeta)q^+_{j-1}(e^\zeta)q^+_{j+1}(e^\zeta)
\end{equation}
where $q^{\pm}(z)$ are polynomials related to the twisted Wronskians \eqref{Dtrig} in the same way as in Theorem  \ref{Wth}.

One obtains the following result, which is a direct analog of Theorem \ref{qqrop}. Using nondegeneracy conditions as in the rational case, we obtain:

\begin{Thm}
\label{qqtop}
There is a one-to-one correspondence between the set of nondegenerate $Z$-twisted  Miura $SL(r+1)$ opers and the set of nondegenerate polynomial solutions of the qq$^t$-system \eqref{eq:QQGoperttr}.
\end{Thm}

The explicit relation between the roots of $q^i_+(z)$ in this nondegenerate case can be summarized in the following statement.  

\begin{Thm}\label{Th:QQtGaud}
There is a bijection between the nondegenerate solutions of the trigonometric Gaudin Bethe Ansatz equations 
\begin{equation}
\label{eq:tGaudinBethe}
\frac{\gamma_{i+1}-\gamma_i}{s_{i,a}}-\sum_{(j,b)\neq(i,a)}\frac{a_{ij}}{s_{i,a}-s_{j,b}}+\sum_{c=1}^{M_i}\frac{1}{s_{i,a}-a_{i,c}}=0
\end{equation}
for $i=1,\dots, r+1$, where $a_{ji}$ is the Cartan matrix,
and the nondegenerate polynomial solutions of the qq$^t$-system \eqref{eq:QQGoperttr}.
where $\Lambda_j(z)=\prod_{c=1}^{M_j}(z-a_{j,c})$ and $q_j^+(z)=\prod_{c=1}^{N_j}(z-s_{j,c})$
\end{Thm}

\begin{proof}
Notice that
\begin{equation}
\label{eq:ResidueCalc}
\partial_\zeta\left[e^{(\gamma_{i+1}-\gamma_i) \zeta}\frac{q_i^-(e^\zeta)}{q_i^+(e^\zeta)}\right] = \frac{e^{(\gamma_{i+1}-\gamma_i) \zeta}}{q_i^+(e^\zeta)^2}\left((\gamma_{i+1}-\gamma_i) q_i^+(e^\zeta)q_i^-(e^\zeta)+q_i^+(e^\zeta)\partial_\zeta q_i^-(e^\zeta)-q_i^-(e^\zeta)\partial_\zeta q_i^+(e^\zeta)\right),
\end{equation}
where in the right-hand side above in the parentheses we can find the left-hand side of the qq$^t$-equations \eqref{eq:QQGoperttr}. Recall that a meromorphic function $f(u)$ with a double pole at $v$ does not have a residue iff $\partial_u \log(f(u)(u-v)^2)\vert_{u=v}=0$.
Let $\Lambda_j(e^\zeta)=\prod_{c=1}^{M_j}(e^\zeta-a_{j,c})$ and $q_j^+(e^\zeta)=\prod_{c=1}^{N_j}(e^\zeta-s_{j,c})$.
Therefore, using \eqref{eq:QQGoperttr}, we can write the following
\begin{align}
&\partial_\zeta \log\left[e^{(\gamma_{i+1}-\gamma_i) \zeta}q_i^+(e^\zeta)^{-2}\Lambda_i(e^\zeta)\prod_{j\neq i}\left[q^+_{j}(e^\zeta)\right]^{-a_{ji}}(e^\zeta-s_{i,a})^2\right]\Bigg\vert_{e^\zeta=s_{i,b}}\cr
&=\gamma_{i+1}-\gamma_i+\partial_\zeta \log\left[\Lambda_i(e^\zeta)\prod_{j}\left[q^+_{j}(e^\zeta)\right]^{-a_{ji}}(e^\zeta-s_{i,a})^2\right]\Bigg\vert_{e^\zeta=s_{i,b}}=0\,,
\end{align}
where we also used $a_{ii}=2$. The calculation of the logarithmic derivative yields
\begin{equation}
\gamma_{i+1}-\gamma_i+\sum_{c=1}^{M_i}\frac{s_{i,b}}{s_{i,b}-a_{i,c}}-s_{i,b}\cdot\sum_{(j,c)\neq(i,b)}\frac{a_{ij}}{s_{i,b}-s_{j,c}}=0\,,
\end{equation}
which is equivalent to \eqref{eq:tGaudinBethe}.
\end{proof}

Finally, let us define the spaces and the space of functions on them $tOp^Z_{\Lambda}$, ${\rm Fun}(tOp^Z_{\Lambda})$ in exactly the same way we did it with $rOp^Z_{\Lambda}$, ${\rm Fun}(rOp^Z_{\Lambda})$ except that rational $Z$-twist condition is replaced by the trigonometric one. This means that 
\begin{eqnarray}
{\rm Fun}\big(t{\rm Op}_Z^{\Lambda}\big)=\frac{\mathbb{C}(\gamma_i, p_i, a_i)}{\rm Wr},
\end{eqnarray}
where the relation ${\rm Wr}$ stands for:
\begin{equation}
 \underset{i,j}{\det} \left[e^\zeta(1+\gamma_i)^{j-1}-p_i \gamma_i^{j-1}\right] = \Lambda(z),
\label{eq:difMiuraDetFormdis}
\end{equation}

Proceeding analogously to the previous section we arrive at the following theorem

\begin{Thm}\label{Th:tGaudin/rRS}

There is an isomorphism of algebras 
\begin{equation}
{\rm Fun}\big(t{\rm Op}_Z^{\Lambda}\big)\cong \frac{\mathbb{C}( \gamma_i, p_i, a_i)}
{\{ H^{rRS}_k=e_k(a_1, \dots, a_{r+1})\}_{k=1,\dots, r}},
\end{equation}
where 
\begin{equation}\label{eq:rRSW}
\text{det}\Big(z - t \Big)=\sum_k H^{rRS}_k( \{\xi_i\}, \{p_i\})z^k
\end{equation}
where  $t$ is the rRS Lax matrix $\{t_{ij}\}_{i,j=1,\dots, r+1}$:
\begin{align}
\label{eq:LaxrRScalc}
t_{ij}&=\sum_k p_i \gamma_i^{k-1}\cdot (-1)^{kj}\frac{S_{r-k,j}(1+\gamma_1,\dots,1+\gamma_r)}{\prod\limits_{l\neq j}^{r+1} (\gamma_j-\gamma_l)}\notag\\
&=\frac{\prod\limits_{m \neq j}^{r+1} \left(\gamma_i - \gamma_m-1 \,  \right) }{\prod\limits_{l\neq j}^{r+1}(\gamma_j-\gamma_l)}   p_i\,,
\end{align}
\end{Thm}

\begin{proof}
The proof goes along the lines of the proof of Theorem \ref{Th:XXz/tRS} where for the Vandermonde matrix we need to use
$\left(V_{1,\dots,r}\right)_{ij}=(1+\gamma_i)^{j-1}$. The resulting rRS Lax matrix reads
\begin{equation}
t_{ij}=-M_{1,\dots,r}(-\infty)\cdot V_{1,\dots,r}^{-1}\,,
\end{equation}
where $\left(M_{1,\dots,r}\right)_{ij} = e^\zeta(1+\gamma_i)^{j-1}-p_i \gamma_i^{j-1}$ and we take $\left(V_{1,\dots,r+1}\right)_{ij}=(1+\gamma_i)^{j-1}$ for the Vandermonde matrix and $\beta_{r+1} = \det V_{1,\dots,r+1}$.
Performing matrix multiplication as in the proof of Theorem \ref{Th:XXz/tRS} we obtain the following
\begin{align}
\label{eq:LaxrRScalc2}
t_{ij}&=\sum_k p_i \gamma_i^{k-1}\cdot (-1)^{kj}\frac{S_{r-k,j}(1+\gamma_1,\dots,1+\gamma_r)}{\prod\limits_{l\neq j}^{r+1} (\gamma_j-\gamma_l)}\notag\\
&=\frac{\prod\limits_{m \neq j}^{r+1} \left(\gamma_i - \gamma_m-1 \,  \right) }{\prod\limits_{l\neq j}^{r+1}(\gamma_j-\gamma_l)}   p_i\,,
\end{align}
which is the Lax matrix for the rational Ruijsenaars-Schneider model.
\end{proof}

Thus we arrive to the following statement. 

\begin{Cor}(Rational differential Quantum/Classical duality)\label{rdqc2}
We have the following isomorphisms:
\begin{eqnarray}
\frac{\mathbb{C}( \gamma_i, s_{k,l}, a_i)}{\rm Bethe}\cong {\rm Fun}\big(t{\rm Op}_Z^{\Lambda}\big)\cong \frac{\mathbb{C}( \gamma_i, p_i, a_i)}
{( H^{rRS}_k=e_k(a_1, \dots, a_{r+1}))_{k=1,\dots, r}},
\end{eqnarray}
where in the first term {\rm Bethe} stands for the trigonometric Gaudin Bethe equations from \eqref{eq:tGaudinBethe} specified for $tOp^Z_{\Lambda}$ and 
the latter space can be thought of as the space of functions on the intersection $\mathscr{L}_1\cap \mathscr{L}_2$ of two Lagrangian subvarieties:
\begin{eqnarray} 
\mathscr{L}_1=\{x^i=\xi_i\}_{i=1,\dots, r+1}, \quad \mathscr{L}_2=\{H^{rRS}_i=e_i(\{a_i\})\}_{i=1,\dots, r+1}.
\end{eqnarray}
\end{Cor}

\section{The Calogero-Moser Space and its Dualities}\label{Sec:CMSpace}

\subsection{The tRS Model}
Consider the subset of $GL(r+1;\mathbb{C})\times GL(r+1;\mathbb{C})\times\mathbb{C}^{r+1}\times \mathbb{C}^{r+1}$ defined by the relation
\begin{equation}
\label{eq:flatnesscond}
q M T - TM = u \otimes v^T\,,
\end{equation}
which is subject to the group action
\begin{equation}
(M,T,u,v)\mapsto (g M g^{-1},g T g^{-1},g u, v g^{-1}),\qquad g\in GL(r+1;\mathbb{C})\,.
\end{equation}
The quotient of the subset by this action yields the Calogero-Moser space $\mathcal{M}$ \cite{Oblomkov:aa}.

Solutions of \eqref{eq:flatnesscond} for $T$ (or $M$) in the basis where $M$ (or $T$) is diagonal lead to tRS Lax matrices. 
Here we discuss limiting cases of tRS into rRS and tCM.

\vskip.1in

The Lax matrix for the tRS matrix can be constructed as follows. 
\begin{Lem}[\cite{Koroteev:2023aa}]
Let $M$ and $T$ satisfy \eqref{eq:flatnesscond}. In the basis where $M$ is a diagonal matrix with eigenvalues $\xi_1,\dots, \xi_{r+1}$ the components of matrix $T$ are given by the following expression:
\begin{equation} \label{eq:lax1}
T_{ij} = \frac{u_i v_j}{q \xi_i - \xi_j}\,.
\end{equation}
\end{Lem}

One can define the {\it tRS momenta} $p_i,\,i=1,\dots,r+1$ using the diagonal components as follows:
\begin{equation}\label{eq:lax2}
p_i = - u_i v_i \frac{\prod\limits_{k \neq i} (\xi_i - \xi_k)}{\prod\limits_{k}\left(\xi_i -\xi_k q\right)}\,.
\end{equation}

Using the above formula we can represent the components of matrix $T$ \eqref{eq:lax1} by properly scaling vectors $u$ and $v$:
\begin{equation}\label{eq:LaxFullFormula}
T_{ji}= \frac{\xi_j(1 - q)}{\xi_j - \xi_i q}\prod\limits_{k \neq j} \frac{\xi_j-\xi_k q}{\xi_j - \xi_k}\, p_j 
=\frac{\prod\limits_{k \neq i}(\xi_j-\xi_k q)}{\prod\limits_{k \neq j}(\xi_j - \xi_k)}\, p_j \,.
\end{equation}
Matrix $T$ \eqref{eq:LaxFullFormula} is known as the Lax matrix of the tRS model \cite{MR1329481} (see also \cite{Gorsky:1993dq,Fock:1999ae}).

The coefficients of the characteristic polynomial of the Lax matrix are the tRS Hamiltonians $H_k$
\begin{equation}
 \text{det}\left(z -  T(\xi_i, p_i,q) \right) = \sum_{k=0}^{r+1} (-1)^lH_k(\xi_i, p_i,q) ^{n-k}\,,
\label{eq:tRSLaxDecomp}
\end{equation}

The corresponding tRS integrals of motion can be obtained by equating the above characteristic polynomials to $\prod_i (z-a_i)$ 
\begin{equation}
\sum_{\substack{\mathcal{I}\subset\{1,\dots,L\} \\ |\mathcal{I}|=k}}\prod_{\substack{i\in\mathcal{I} \\ j\notin\mathcal{I}}}\frac{q\,\xi_i - \xi_j }{\xi_i-\xi_j}\prod\limits_{m\in\mathcal{I}}p_m = e_k (a_i)\,,
\end{equation}
where $e_k$ is the $k$th elementary symmetric function.

\subsubsection{3d mirror Symmetry}
The defining relation \eqref{eq:flatnesscond} is invariant under the following symmetry
\begin{equation}
\label{eq:TMqmirror}
q\mapsto q^{-1}\,,\qquad M\mapsto T\,,\qquad T\mapsto M\,.
\end{equation}
It was shown in \cite{Gaiotto:2013bwa,Koroteev:2023aa} that this symmetry is related to the 3d mirror symmetry of the dual quiver varieties. Indeed, under the quantum/classical duality the eigenvalues of $M$ are related to the maximal torus parameters $a_i$ while the eigenvalues of $T$ yield the K\"ahler parameters $z_i$ for the underlying quiver variety. The 3d mirror symmetry interchanges the two sets and in addition, inverts the parameter $q$ which describes dilations of the cotangent directions.

\vskip.1in

Let us perform the 3d mirror map \eqref{eq:TMqmirror} upon which the twist and singularity parameters $\{a_i\}$ and $\{\xi_i\}$ are interchanged as well as $q$ is replaced with $q^{-1}$. Since the relation \eqref{eq:flatnesscond} is invariant under this map we have the following

\begin{Prop}\label{tqqcm2}
The following algebras are isomorphic
\begin{eqnarray}
\frac{\mathbb{C}(q, \xi_i, p_i, a_i)}{( H^{tRS}_k=e_k(a_1, \dots, a_{r+1}))_{k=1,\dots, r+1}}\cong 
\frac{\mathbb{C}(q^{-1}, a_i, p'_i, \xi_i)}{( H'^{tRS}_k=e_k(\xi_1, \dots, \xi_{r+1}))_{k=1,\dots, r+1}}
\end{eqnarray}
Here the dual Hamiltoninas $H'^{tRS}_k$ are obtained from the 3d mirror dual Lax matrix $M$ and are given by
\begin{equation}
H'^{tRS}_k = \sum_{\substack{\mathcal{I}\subset\{1,\dots,L\} \\ |\mathcal{I}|=k}}\prod_{\substack{i\in\mathcal{I} \\ j\notin\mathcal{I}}}\frac{q^{-1}\,a_i - a_j }{a_i-a_j}\prod\limits_{m\in\mathcal{I}}p'_m\,.
\end{equation}
\end{Prop}

Let us define dual twist parameters $Z'=\text{diag}(a_1,\dots,a_{r+1})$ as well as singularities in the dual frame given by polynomial $\Lambda'(z)=\prod_{i=1}^{r+1}(z-\xi_i)$. These data describe the space of dual q-opers ${\rm Fun}\big(q{\rm Op}_{Z'}^{\Lambda'}\big)$.

Using Corollary \ref{tqqc} as well as the above proposition we obtain the following statement

\begin{Thm}
The space of canonical trigonometrically-twisted q-opers on $\mathbb{P}^1$ is self-dual, namely:
\begin{eqnarray}
{\rm Fun}\big(q{\rm Op}_{Z'}^{\Lambda'}\big)\cong {\rm Fun}\big(q{\rm Op}_{Z}^{\Lambda})\,.
\end{eqnarray}
\end{Thm}

\vskip.1in
The limiting cases of the rRS, tCM models, as we shall see below, can also be related to each other via the 3d mirror symmetry, since they can be obtained as dual limits of \eqref{eq:flatnesscond}. Finally, the rRS model, which is at the bottom of the hierarchy, is self-dual under 3d mirror symmetry as it can be obtained in two different limits from the rRS and from the tCM models.

\subsection{The tCM Model}
Let $M$ be the matrix exponent $M = \exp (R m)$ and $q = e^{R\epsilon }$. Then we send $R\rightarrow 0$ to get from \eqref{eq:flatnesscond}
\begin{equation}
\label{eq:flatnesscondtCM}
[m, T]+\epsilon T = \tilde u \otimes \tilde v^T\,,
\end{equation}
where we scaled $u, v$ by $R$.

In the basis where $T=\text{diag}(\zeta_1,\dots,\zeta_{r+1})$ matrix $m$ becomes the Lax matrix of the tCM model. This works as follows. The diagonal components of the Lax matrix read $\epsilon \zeta_i = \tilde u_i  \tilde v_i$. Then the off-diagonal components of $m$ are
\begin{equation}
\label{eq:tCMLax1}
m_{ij}=\frac{\tilde u_i  \tilde v_j}{\zeta_i-\zeta_j}\,,
\end{equation}
while the diagonal components of $m$ will contain the tCM momenta
\begin{equation}
\label{eq:tCMLax2}
m_{ii}= p_i +\epsilon\sum_{j\neq i}\frac{\zeta_i+\zeta_j}{\zeta_i-\zeta_j}\,.
\end{equation}
The choice of vector $\tilde v_j$ in \eqref{eq:tCMLax1} is given by $\tilde v_j=\prod_{k\neq j}(\zeta_j-\zeta_k)^{-1}$
$$
m_{ij}=\frac{\epsilon\zeta_i}{\zeta_i-\zeta_j}\frac{\tilde v_j}{\tilde v_i}=\frac{\epsilon\zeta_i}{\zeta_i-\zeta_j}\frac{\prod\limits_{k\neq i}(\zeta_i-\zeta_k)}{\prod\limits_{k\neq j}(\zeta_j-\zeta_k)}\,,
$$
which matches \eqref{eq:tCMLaxm}.

\subsection{The rRS Model}
Now consider \eqref{eq:flatnesscondtCM} and matrix $m$ to be diagonal with eigenvalues $\gamma_1,\dots, \gamma_{N}$. In this case matrix $T$ becomes the Lax matrix of the rational Ruijsenaars-Schneider model. Its matrix elements are
\begin{equation}
T_{ij}=\frac{u_i v_j}{\gamma_i-\gamma_j+\epsilon}
\end{equation}
and the rRS momenta can be defined via the diagonal components as a limit from \eqref{eq:lax2}
\begin{equation}
u_i v_i = -p_i\frac{\prod\limits_{k=1}^{r+1}(\gamma_i-\gamma_k-\epsilon)}{\prod\limits_{k\neq i}^{r+1}(\gamma_i-\gamma_k)}\,.
\label{eq:uvdiagonal}
\end{equation}
In other words, the Lax matrix, after setting $v_j=1$ for all $j$, becomes
\begin{equation}
\label{eq:rRSLaxApp}
T_{ij} = \frac{\prod\limits_{k\neq j}^{r+1}(\gamma_i-\gamma_k-\epsilon)}{\prod\limits_{k\neq i}^{r+1}(\gamma_i-\gamma_k)}p_i\,,
\end{equation}
which coinsides with \eqref{eq:LaxrRScalc} up to rescaling by $\epsilon$.

\vskip.1in

There is an equivalence between the rRS and the tCM models since their Lax matrices can be found from the diagonalization of \eqref{eq:flatnesscondtCM} $T$ and $m$ respectively. 
\begin{Prop}
The following algebras are isomorphic:
\begin{eqnarray}
\frac{\mathbb{C}(\gamma_i, p_i, a_i)}{( H^{rRS}_k=e_k(a_1, \dots, a_{r+1}))_{k=1,\dots, r+1}}\cong 
\frac{\mathbb{C}(a_i, p'_i, \gamma_i)}{( H^{tCM}_k=e_k(\xi_1, \dots, \xi_{r+1}))_{k=1,\dots, r+1}}.
\end{eqnarray}
\end{Prop}

Notice that parameter $\epsilon$ can be removed completely from \eqref{eq:flatnesscondtCM} by scaling $m$ and the right-hand side. That is why it did not appear in the algebras in the above Proposition. We nevertheless keep $\epsilon$ manifest in this section in order to be able to take the limit to the rational Calogero-Moser system below.

Using Corollaries \ref{rteo} and \ref{rdqc2} we can prove the following
\begin{Thm}
There is a one-to-one correspondence between the space of canonical trigonometrically $Z$-twisted opers on $\mathbb{P}^1$ and the space of canonical rationally $Z'$-twisted $\epsilon$-opers on $\mathbb{P}^1$
\begin{equation}
{\rm Fun}\big(t{\rm Op}_Z^{\Lambda}\big)(\mathbb{C}(\epsilon))\cong {\rm Fun}\big(\epsilon{\rm Op}_{Z'}^{\Lambda'}\big)\,,
\end{equation}
where the singularities in the above spaces are defined using polynomials $\Lambda(z)=\prod_{i=1}^{r+1}(z-\gamma_i)$ and $\Lambda'(z)=\prod_{i=1}^{r+1}(z-a_i)$ respectively.
\end{Thm}

\subsection{The rCM Model}
Let $T$ be the matrix exponent $T = \exp \epsilon t$. Then we send $\epsilon\to 0$ to get from \eqref{eq:flatnesscondtCM}
in the first order in $\epsilon$
\begin{equation}
\label{eq:flatnesscondrCM}
[m, t] + 1 =  u' \otimes  v'\,,
\end{equation}
for some vectors $u'$ and $v'$. In the basis where $m=\text{diag}(m_1,\dots,m_{r+1})$ the diagonal elements of the left hand side vanish, thus $u'_iv_i'=1$ whereas the off-diagonal elements of $t$ read
\begin{equation}
\label{eq:LaxrCMoffdiag1}
t_{ij}=\frac{u'_iv_j'}{m_i-m_j}\,.
\end{equation}
We can set 
\begin{equation}
u_i'=\prod_{k\neq {j}}(\alpha_i-\alpha_k)\,,\qquad v_j'=\prod_{l\neq i}(\alpha_j-\alpha_l)^{-1}
\end{equation}
in order to match the expression with \eqref{eq:rCMoffdiag}. The diagonal components can be obtained by taking the limit from \eqref{eq:uvdiagonal}
\begin{equation}
\label{eq:LaxrCMdiag1}
t_{ii}=P_i-\sum_{j\neq i}\frac{1}{m_i-m_j}\,,
\end{equation}
where $\mathrm{p}_i=e^{\epsilon P_i}$, we send $\epsilon\to 0$ and capture the term linear in $\epsilon$. We thus have obtained the Lax matrix $t$ of the rCM model.

\vskip.1in

Since the relation \eqref{eq:flatnesscondrCM} is invariant under the 3d mirror map 
\begin{equation}
m\mapsto t\,,\qquad t\mapsto -m\,
\end{equation}
we can formulate

\begin{Prop}\label{rdqc3}
The following algebras are isomorphic:
\begin{eqnarray}
\frac{\mathbb{C}( \gamma_i, p_i, a_i)}
{( H^{rCM}_k=e_k(a_1, \dots, a_{r+1}))_{k=1,\dots, r}}\cong
\frac{\mathbb{C}( a_i, p'_i, \gamma_i)}
{( H'^{rCM}_k=e_k(\gamma_1, \dots, \gamma_{r+1}))_{k=1,\dots, r}}\,.
\end{eqnarray}
where the dual rCM Hamiltonians in terms of variable $a_i$ are used.
\end{Prop}

Let us define dual twist parameters $Z'=\text{diag}(a_1,\dots,a_{r+1})$ as well as singularities in the dual frame given by polynomial $\Lambda'(z)=\prod_{i=1}^{r+1}(z-\xi_i)$. Then Corollary \ref{rdqc} and the above proposition yield the following statement.

\begin{Thm}
The space of canonical rationally $Z$-twisted opers ${\rm Fun}\big(r{\rm Op}_Z^{\Lambda}\big)$ is self-dual, namely:
\begin{eqnarray}
{\rm Fun}\big(r{\rm Op}_Z^{\Lambda}\big)\cong {\rm Fun}\big(r{\rm Op}_{Z'}^{\Lambda'}\big)\,.
\end{eqnarray}
\end{Thm}

\bibliography{cpn1}
\end{document}